\documentclass{birkjour}
\usepackage{hyperref}

 \newtheorem{theorem}{Theorem}[section]
 \newtheorem{corollary}[theorem]{Corollary}
 \newtheorem{lemma}[theorem]{Lemma}
 \numberwithin{equation}{section}

\newcommand{\cA}{\mathcal{A}}
\newcommand{\cB}{\mathcal{B}}

\newcommand{\cE}{\mathcal{E}}
\newcommand{\cF}{\mathcal{F}}
\newcommand{\cI}{\mathcal{I}}
\newcommand{\cJ}{\mathcal{J}}
\newcommand{\cK}{\mathcal{K}}
\newcommand{\cM}{\mathcal{M}}
\newcommand{\cZ}{\mathcal{Z}}

\newcommand{\mC}{\mathbb{C}}
\newcommand{\mN}{\mathbb{N}}
\newcommand{\mR}{\mathbb{R}}
\newcommand{\mZ}{\mathbb{Z}}

\newcommand{\fa}{\mathfrak{a}}
\newcommand{\fb}{\mathfrak{b}}
\newcommand{\fc}{\mathfrak{c}}

\newcommand{\ff}{\mathfrak{f}}

\newcommand{\fr}{\mathfrak{r}}
\newcommand{\fs}{\mathfrak{s}}

\newcommand{\fA}{\mathfrak{A}}
\newcommand{\fJ}{\mathfrak{J}}
\newcommand{\fS}{\mathfrak{S}}
\newcommand{\fZ}{\mathfrak{Z}}

\newcommand{\alg}{{\rm alg}}
\newcommand{\id}{{\rm id}}
\newcommand{\Co}{{\rm Co}}
\newcommand{\Op}{{\rm Op}}

\begin{document}
\title[Semi-Fredholmness of singular integral operators with shifts]{%
Semi-Fredholmness of Weighted Singular\\
Integral Operators with Shifts and Slowly\\
Oscillating Data}

\author{Alexei Yu. Karlovich}

\address{%
Centro de Matem\'atica e Aplica\c{c}\~oes,\\
Departamento de Matem\'a\-tica, \\
Faculdade de Ci\^encias e Tecnologia,\\
Universidade Nova de Lisboa,\\
Quinta da Torre, \\
2829--516 Caparica, Portugal}
\email{oyk@fct.unl.pt}

\author{Yuri I. Karlovich}

\address{%
Centro de Investigaci\'on en Ciencias,\\
Instituto de Investigaci\'on en Ciencias B\'asicas y Aplicadas,\\
Universidad Aut\'onoma del Estado de Morelos,\\
Av. Universidad 1001, Col. Chamilpa,\\
C.P. 62209 Cuernavaca, Morelos, M\'exico}
\email{karlovich@uaem.mx}

\author{Amarino B. Lebre}
\address{%
Centro de An\'alise Funcional, Estruturas Lineares e Aplica\c{c}\~oes,\\
Departamento de Matem\'atica, \\
Instituto Superior T\'ecnico,\\
Universidade de Lisboa,\\
Av. Rovisco Pais, \\
1049--001 Lisboa, Portugal}
\email{alebre@math.tecnico.ulisboa.pt}

\thanks{%
This work was partially supported by the Funda\c{c}\~ao para a Ci\^encia e a
Tecnologia (Portu\-guese Foundation for Science and Technology)
through the projects
UID/MAT/00297/2013 (Centro de Matem\'atica e Aplica\c{c}\~oes)
and
UID/MAT/04721/2013 (Centro de An\'alise Funcional, Estruturas Lineares e
Aplica\c{c}\~oes).
The second author was also supported by the SEP-CONACYT Project
No. 168104 (M\'exico).}

\begin{abstract}
Let $\alpha,\beta$ be orientation-preserving homeomorphisms of $[0,\infty]$
onto itself, which have only two fixed points at $0$ and $\infty$, and whose
restrictions to $\mR_+=(0,\infty)$ are diffeomorphisms, and let
$U_\alpha,U_\beta$ be the corresponding isometric shift operators on the
space $L^p(\mR_+)$ given by $U_\mu f=(\mu')^{1/p}(f\circ\mu)$ for
$\mu\in\{\alpha,\beta\}$. We prove sufficient conditions for the right and
left Fredholmness on $L^p(\mR_+)$ of singular integral operators of the
form $A_+P_\gamma^++A_-P_\gamma^-$, where $P_\gamma^\pm=(I\pm S_\gamma)/2$,
$S_\gamma$ is a weighted Cauchy singular integral operator,
$A_+=\sum_{k\in\mZ}a_kU_\alpha^k$ and $A_-=\sum_{k\in\mZ}b_kU_\beta^k$
are operators in the Wiener algebras of functional operators with shifts.
We assume that the coefficients $a_k,b_k$ for $k\in\mZ$ and the derivatives
of the shifts $\alpha',\beta'$ are bounded continuous functions on
$\mR_+$ which may have slowly oscillating discontinuities at $0$ and $\infty$.
\end{abstract}

\keywords{Right Fredholmness, left Fredholmness,
slowly oscillating shift,
Wiener algebra of functional operators,
weighted singular integral operator,
Mellin pseudodifferential operator.}

\subjclass{45E05, 47A53, 47G10, 47G30.}
\maketitle

%%%----------------------------------------------------------------------------
\section{Introduction}
Let $\cB(X)$ denote the Banach algebra of all bounded linear operators
acting on a Banach space $X$. Recall that an operator $A\in\cB(X)$ is said
to be left invertible (resp. right invertible) if there
exists an operator $B\in\cB(X)$ such that $BA=I$ (resp. $AB=I$) where
$I\in\cB(X)$ is the identity operator on $X$. The operator $B$ is called a
left (resp. right) inverse of $A$. An operator $A\in\cB(X)$ is said to be
invertible if it is left invertible and right invertible simultaneously.
We say that $A$ is strictly left (resp. right) invertible if it is left
(resp. right) invertible, but not invertible. If the operator $A$ is
invertible only from one side, then the corresponding inverse is not uniquely
defined. We refer to \cite[Section~2.5]{GK92} for further properties of
one-sided invertible operators acting on Banach spaces.

Let $\cK(X)$ be the closed two-sided ideal of all compact operators in
$\cB(X)$, and let $\cB^\pi(X):=\cB(X)/\cK(X)$ be the Calkin algebra of
the cosets $A^\pi:=A+\cK(X)$ where $A\in\cB(X)$. Following
\cite[Chap.~XI, Definition~2.3]{C90}, an operator $A\in\cB(X)$ is said to
be left Fredholm (resp., right Fredholm) if the coset $A^\pi$ is left
invertible (resp., right invertible) in the Calkin algebra $\cB^\pi(X)$. An
operator $A\in\cB(X)$ is said to be semi-Fredholm if it is left or right
Fredholm. We will write $A\simeq B$ if $A-B\in\cK(X)$.

Let $C_b(\mR_+)$ denote the $C^*$-algebra of all bounded continuous
functions on the positive half-line $\mR_+:=(0,+\infty)$. Following Sarason
\cite[p.~820]{S77}, a function $f\in C_b(\mR_+)$ is called
slowly oscillating (at $0$  and $\infty$) if
\[
\lim_{r\to s}\sup_{t,\tau\in[r,2r]}|f(t)-f(\tau)|=0
\quad\mbox{for}\quad
s\in\{0,\infty\}.
\]
The set $SO(\mR_+)$ of all slowly oscillating functions is
a $C^*$-algebra. This algebra properly contains $C(\overline{\mR}_+)$,
the $C^*$-algebra of all continuous functions on the two-point
compactification $\overline{\mR}_+:=[0,+\infty]$ of $\mR_+$.

Suppose $\alpha$ is an orientation-preserving homeomorphism of
$\overline{\mR}_+$ onto itself, which has only two fixed points $0$ and
$\infty$, and whose restriction to $\mR_+$ is a diffeomorphism. We say
that $\alpha$ is a slowly oscillating shift if $\log\alpha'$ is bounded
and $\alpha'\in SO(\mR_+)$. The set of all slowly oscillating shifts is
denoted by $SOS(\mR_+)$. By \cite[Lemma~2.2]{KKL11a}, an
orientation-preserving diffeomorphism $\alpha$ of $\mR_+$ onto itself belongs
to $SOS(\mR_+)$ if and only if $\alpha(t)=te^{\omega (t)}$ for $t\in\mR_+$
and a real-valued function $\omega\in SO(\mR_+)\cap C^1(\mR_+)$
is such that the function $\psi$ given by
$\psi(t):= t\omega^\prime(t)$ also belongs to $SO(\mR_+)$ and
$\inf_{t\in\mR_+}\big(1+t\omega'(t)\big)>0$. The real-valued slowly
oscillating function
\[
\omega(t):=\log[\alpha(t)/t],
\quad
t\in\mR_+,
\]
is called the exponent function of $\alpha\in SOS(\mR_+)$.

Through the paper, we will suppose that $1<p<\infty$ and will use the
following notation:
\[
\cB:=\cB(L^p(\mR_+)),
\quad
\cK:=\cK(L^p(\mR_+)).
\]
It is easily seen that if $\alpha\in SOS(\mR_+)$, then the weighted shift
operator defined by
\[
U_\alpha f:=(\alpha')^{1/p}(f\circ\alpha)
\]
is  an isometric isomorphism of the Lebesgue space $L^p(\mR_+)$
onto itself. It is clear that $U_\alpha^{-1}=U_{\alpha_{-1}}$,
where $\alpha_{-1}$ is the inverse function to $\alpha$.
For $k\in\mN$, we denote by $U_\alpha^{-k}$ the
operator $(U_\alpha^{-1})^k$. Let $W_{\alpha,p}^{SO}$ denote the
collection of all operators of the form
%%%
\begin{equation}\label{eq:operator-series}
A=\sum_{k\in\mZ}a_k U_\alpha^k
\end{equation}
%%%
where $a_k\in SO(\mR_+)$ for all $k\in\mZ$ and
%%%
\begin{equation}\label{eq:Wiener-norm}
\|A\|_{W_{\alpha,p}^{SO}}:=\sum_{k\in\mZ}\|a_k\|_{C_b(\mR_+)}<+\infty.
\end{equation}
%%%
The set $W_{\alpha,p}^{SO}$ is, actually, a Banach algebra with respect
to the usual operations and the norm \eqref{eq:Wiener-norm}. By analogy with
the Wiener algebra of absolutely convergent Fourier series, we will call
$W_{\alpha,p}^{SO}$ the Wiener algebra.

Let $\Re\gamma$ and $\Im\gamma$ denote the real and imaginary part
of $\gamma\in\mC$, respectively. If $\gamma\in\mC$ satisfies
%%%
\begin{equation}\label{eq:gamma-condition}
0<1/p+\Re\gamma<1,
\end{equation}
%%%
then the operator
%%%
\begin{equation}\label{eq:def-S}
(S_\gamma f)(t):=\frac{1}{\pi i}\int_{\mR_+}
\left(\frac{t}{\tau}\right)^\gamma\frac{f(\tau)}{\tau-t}d\tau,
\end{equation}
%%%
where the integral is understood in the principal value sense,
is bounded on the Lebesgue space $L^p(\mR_+)$
(see, e.g., \cite[Proposition~4.2.11]{RSS11}).
Put
\[
P_\gamma^\pm:=(I\pm S_\gamma)/2.
\]

This paper is a continuation of our recent works
\cite{FTK-BJMA,FTK,KKL-JMAA,KKL-JIEA} (see also references therein).
Let $\alpha,\beta$ belong to $SOS(\mR_+)$ and $a_k,b_k\in SO(\mR_+)$
for all $k\in\mZ$. In \cite{KKL-JMAA,KKL-JIEA} we found criteria for the
Fredholmness and a formula permitting to calculate the index of the weighted
singular integral operator of the form
\[
M:=(a_0 I+a_1U_\alpha)P_\gamma^++(b_0 I+b_1U_\beta)P_\gamma^-.
\]
In this paper we assume that
%%%
\begin{equation}\label{eq:def-Apm}
A_+:=\sum_{k\in\mZ}a_k U_\alpha^k\in W_{\alpha,p}^{SO},
\quad
A_-:=\sum_{k\in\mZ}b_k U_\beta^k\in W_{\beta,p}^{SO}
\end{equation}
%%%
and consider the weighted  singular integral operator of the form
%%%
\begin{equation}\label{eq:def-N}
N:=A_+P_\gamma^++A_-P_\gamma^-.
\end{equation}
%%%
Criteria for the Fredholmness of the operator $N$ in the
particular case of $\alpha=\beta$ and $\gamma=0$ were obtained in
\cite{FTK}. The proof of the sufficiency portion is based on the
Allan-Douglas local principle and follows ideas of \cite{KKL11a}.
In this paper we will show that the localization technique is flexible enough
to treat also the case of the left and right Fredholmness for arbitrary shifts
$\alpha,\beta$ and arbitrary $\gamma$ satisfying \eqref{eq:gamma-condition},
provided that there are one-sided inverses of $A_+$ and $A_-$
belonging to the Wiener algebras $W_{\alpha,p}^{SO}$ and $W_{\beta,p}^{SO}$,
respectively. We show that the required result on one-sided inverses can be
obtained from \cite{FTK-BJMA}.

By $M(\mathfrak{A})$ we denote the maximal ideal space of a unital
commutative Banach algebra $\mathfrak{A}$. Identifying the points
$t\in\overline{\mR}_+$ with the evaluation functionals $t(f)=f(t)$
for $f\in C(\overline{\mR}_+)$, we get
$M(C(\overline{\mR}_+))=\overline{\mR}_+$. Consider the fibers
\[
M_s(SO(\mR_+)):=\big\{\xi\in M(SO(\mR_+)):\xi|_{C(\overline{\mR}_+)}=s\big\}
\]
of the maximal ideal space $M(SO(\mR_+))$ over the points
$s\in\{0,\infty\}$. By \cite[Proposition~2.1]{K08}, the set
\[
\Delta:=M_0(SO(\mR_+))\cup M_\infty(SO(\mR_+))
\]
coincides with ${\rm clos}_{SO^*}\mR_+\setminus\mR_+$, where
${\rm clos}_{SO^*}\mR_+$ is the weak-star closure of $\mR_+$ in the dual
space of $SO(\mR_+)$. Then $M(SO(\mR_+))=\Delta\cup\mR_+$. In what follows
we write $a(\xi):=\xi(a)$ for every $a\in SO(\mR_+)$ and every
$\xi\in\Delta$.

With the operators $A_\pm$ defined by \eqref{eq:def-Apm}, we associate the
functions $a_\pm$ defined on $\mR_+\times\mR$ by
%%%
\begin{equation}\label{eq:def-apm}
a_+(t,x):=\sum_{k\in\mZ}a_k(t)e^{ik\omega(t)x},
\quad
a_-(t,x):=\sum_{k\in\mZ}b_k(t)e^{ik\eta(t)x},
\end{equation}
%%%
where $\omega,\eta\in SO(\mR_+)$ are the exponent
functions of $\alpha,\beta$, respectively. Since the series in
\eqref{eq:def-Apm} converge absolutely, we have $a_\pm(\cdot,x)\in SO(\mR_+)$
for all $x\in\mR$. With the operator $N$ we associate the function $n$
defined on $\mR_+\times\mR$ by
\[
n(t,x)=a_+(t,x)p_\gamma^+(x)
+
a_-(t,x)p_\gamma^-(x),
\]
where
%%%
\begin{equation}\label{eq:def-p-gamma-pm}
p_\gamma^\pm(x):=(1\pm s_\gamma(x))/2,
\quad
s_\gamma(x):=\coth[\pi(x+i/p+i\gamma)],
\quad
x\in\mR.
\end{equation}
%%%
Since $a_\pm(\cdot,x), n(\cdot,x)\in SO(\mR_+)$ for every $x\in\mR$,
taking the Gelfand  transform of $n(\cdot,x)$, we obtain for
$(\xi,x)\in (\Delta\cup\mR_+)\times\mR$,
%%%
\begin{equation}\label{eq:def-n}
n(\xi,x):=a_+(\xi,x)p_\gamma^+(x)
+
a_-(\xi,x)p_\gamma^-(x),
\end{equation}
%%%
which gives extensions of the functions $n(\cdot,x)$ to $M(SO(\mR_+))$.
%%%----------------------------------------------------------------------------
\begin{theorem}[Main result]\label{th:one-sided-Fredholmness}
Let $1<p<\infty$ and let $\gamma\in\mC$ satisfy \eqref{eq:gamma-condition}.
Suppose $a_k,b_k\in SO(\mR_+)$ for all $k\in\mZ$ and
$\alpha,\beta\in SOS(\mR_+)$. If
\begin{enumerate}
\item[{\rm (i)}] the functional operators
\[
A_+:=\sum_{k\in\mZ}a_k U_\alpha^k\in W_{\alpha,p}^{SO},
\quad
A_-:=\sum_{k\in\mZ}b_k U_\beta^k\in W_{\beta,p}^{SO}
\]
are left (resp., right) invertible on the space $L^p(\mR_+)$;

\item[{\rm (ii)}]
for every $\xi\in\Delta$, the function $n$ defined by
\eqref{eq:def-apm}--\eqref{eq:def-n} satisfies the inequality
%%%
\begin{equation}\label{eq:n-nondegeneracy}
\inf_{x\in\mR}|n(\xi,x)|>0;
\end{equation}
\end{enumerate}
then the operator $N=A_+P_\gamma^++A_-P_\gamma^-$ is left (resp., right)
Fredholm on the space $L^p(\mR_+)$.
\end{theorem}
%%%----------------------------------------------------------------------------
We conjecture that conditions (i) and (ii) of
Theorem~\ref{th:one-sided-Fredholmness} are also necessary for the one-sided
Fredholmness of the operator $N$.

The paper is organized as follows.
Section~\ref{sec:auxiliary} contains some auxiliary results.
In Section~\ref{sec:Wiener-FO}, on the basis of
recent results from \cite{FTK-BJMA}, we show that if an operator
$A\in W_{\alpha,p}^{SO}$ is left (resp., right) invertible, then at least
one of its left (resp., right) inverses belongs to the same algebra
$W_{\alpha,p}^{SO}$.

Section~\ref{sec:algebra-A} is
devoted to the algebra $\cA$ generated by the identity operator $I$ and the
operator $S_0$. We recall that $\cA$ is the smallest Banach subalgebra of
$\cB$ that contains all operators similar to Mellin convolution operators
with continuous symbols. In particular, the algebra $\cA$ contains the
operator $S_\gamma$ and the operator $R_\gamma$ with fixed singularities
defined by
%%%
\begin{equation}\label{eq:def-R}
(R_\gamma f)(t):=\frac{1}{\pi i}\int_{\mR_+}
\left(\frac{t}{\tau}\right)^\gamma\frac{f(\tau)}{\tau+t}d\tau,
\end{equation}
%%%
where the integral is understood in the principal value sense. We also recall
the description of the maximal ideal space of the algebra
$\cA^\pi:=(\cA+\cK)/\cK$.

In Section~\ref{sec:localization}, we recall a version of the Allan-Douglas
local principle suitable for the study of one-sided invertibility in
subalgebras of the Calkin algebra $\cB^\pi=\cB/\cK$ (see
\cite[Theorem~1.35(a)]{BS06}). Following \cite[Section~6]{KKL11a}, we consider
the algebra $\cZ$ generated by $\cA\cup\cK$ and the operators of the form
$cR_0$, where $c\in SO(\mR_+)$. We recall that the maximal ideal space of 
$\cZ^\pi:=\cZ/\cK$ is homeomorphic to the set 
$\{-\infty,+\infty\}\cup(\Delta\times\mR)$. Further, we consider the algebra
$\Lambda$ of operators of local type that consists of all operators $A\in\cB$
such that $AC-CA\in\cK$ for all $C\in\cZ$. Since $\cZ^\pi$ is a commutative
central subalgebra of $\Lambda^\pi:=\Lambda/\cK$, we can apply the 
Allan-Douglas local principle to $\Lambda^\pi\subset\cB^\pi$ and its central
subalgebra $\cZ^\pi$. In particular, an operator $T\in\Lambda$ is left 
(resp., right) Fredholm if certain cosets $T^\pi+\cJ^\pi_{\xi,x}$,
$T^\pi+\cJ^\pi_{+\infty}$, and $T^\pi+\cJ^\pi_{-\infty}$ are left (resp.,
right) invertible in the corresponding local algebras $\Lambda^\pi_{\xi,x}$,
$\Lambda^\pi_{+\infty}$, and $\Lambda^\pi_{-\infty}$. Here $(\xi,x)$
runs through $\Delta\times\mR$. This result is applicable to the operator $N$
because it belongs to the algebra $\cF_{\alpha,\beta}$ generated by the
operators $S_0$, $U_\alpha^{\pm 1}$, $U_\beta^{\pm 1}$ and the multiplication
operators $cI$ with $c\in SO(\mR_+)$. In turn, this algebra is contained in
the algebra $\Lambda$ of operators of local type.

In Section~\ref{sec:Mellin-PDO}, we recall the definition of the algebra
$\cE(\mR_+,V(\mR))$ of slowly oscillating functions on $\mR_+$ with values
in the algebra $V(\mR)$ of all absolutely continuous functions of finite
total variation. This algebra is important for our purposes because
Mellin pseudodifferential operators with symbols in $\cE(\mR_+,V(\mR))$
commute modulo compact operators. Moreover, if $\alpha\in SOS(\mR_+)$,
then $U_\alpha R_\gamma$ is similar to a Mellin pseudodifferential
operator with symbol in $\cE(\mR_+,V(\mR))$ up to a compact operator.
These results are important ingredients of the proof of two-sided
invertibility of the cosets $N^\pi+\cJ^\pi_{\xi,x}$ in the quotient algebras
$\Lambda_{\xi,x}^\pi$ for $(\xi,x)\in\Delta\times\mR$ under condition (ii)
of Theorem~\ref{th:one-sided-Fredholmness}.

In Section~\ref{sec:proofs}, we prove Theorem~\ref{th:one-sided-Fredholmness}.
Since, according to Section~\ref{sec:Wiener-FO},
there are left/right inverses of $A_+$ (resp., $A_-$) belonging to
$W_{\alpha,p}^{SO}\subset\Lambda$ (resp.,
$W_{\beta,p}^{SO}\subset\Lambda$), the left/right invertibility of
$A_\pm$ implies the left/right invertibility of the coset
$A_\pm^\pi+\cJ^\pi_{\pm\infty}=N^\pi+\cJ^\pi_{\pm\infty}$ in the local
algebra $\Lambda^\pi_{\pm\infty}$. Finally, with the aid of the results
of Section~\ref{sec:Mellin-PDO}, we show that condition (ii) of
Theorem~\ref{th:one-sided-Fredholmness} is sufficient for the two-sided
invertibility of the cosets $N^\pi+\cJ^\pi_{\xi,x}$ in the local algebras
$\Lambda^{\pi}_{\xi,x}$ for all  $(\xi,x)\in\Delta\times\mR$. To complete
the proof of  Theorem~\ref{th:one-sided-Fredholmness}, it remains to
apply the Allan-Douglas local principle (see Section~\ref{sec:localization}).

Finally, in Section~\ref{sec:binomial} we formulate
criteria for the two-sided and one-sided invertibility of a binomial functional
operator with shift in the form $A=aI-bU_\alpha$, which were obtained in
\cite{KKL-MJOM}. These results together with Theorem~\ref{th:one-sided-Fredholmness}
imply more effective sufficient conditions for the left, right, and two-sided
Fredholmness of the operator $(aI-bU_\alpha)P_\gamma^++(cI-dU_\beta)P_\gamma^-$
with $a,b,c,d\in SO(\mR_+)$ and $\alpha,\beta\in SOS(\mR_+)$.
%%%----------------------------------------------------------------------------
\section{Auxiliary results}\label{sec:auxiliary}
%%%----------------------------------------------------------------------------
\subsection{One-sided invertibility of operators on Hilbert spaces}
%%%----------------------------------------------------------------------------
\begin{lemma}\label{le:Hilbert-one-sided}
Let $\mathcal{H}$ be a Hilbert space and $A\in\cB(\mathcal{H})$.
\begin{enumerate}
\item[(a)]
The operator $A$ is left invertible on the space $\mathcal{H}$ if and only if
the operator $A^*A$ is invertible on the space $\mathcal{H}$. In this case,
one of the left inverses of $A$ is given by $A^L=(A^*A)^{-1}A^*$.

\item[(b)]
The operator $A$ is right invertible on the space $\mathcal{H}$ if and only if
the operator $AA^*$ is invertible on the space $\mathcal{H}$. In this case,
one of the right inverses of $A$ is given by 
$A^R=A^*(AA^*)^{-1}$.
\end{enumerate}
\end{lemma}
%%%----------------------------------------------------------------------------
This statement is known, although we are not able to provide a precise
reference. The proof of the sufficiency portion of part (a) is a trivial
computation. Now assume that $\langle\cdot,\cdot\rangle$ is the inner product
of $\mathcal{H}$ and $A^L\in\cB(\mathcal{H})$ is a left inverse of $A$. Then
for every $f\in\mathcal{H}$,
\[
\|f\|^2
\le
\|A^L\|^2\|Af\|^2
=
\|A^L\|^2|\langle Af,Af\rangle|
=
\|A^L\|^2|\langle A^*Af,f\rangle|.
\]
In view of the previous inequality, the invertibility of the operator $A^*A$
follows from the Lax-Milgram theorem (see, e.g.,
\cite[Chap.~III, Section~7]{Y80}). This completes the proof of part (a). The
proof of part (b) is reduced to the previous one by passing to adjoint
operators.

Another proof of the above lemma can be obtained from general results
for $C^*$-algebras contained in \cite[\S~23, Corollaries~2--3]{N72}.

Lemma~\ref{le:Hilbert-one-sided} can also be deduced from more general
results on the Moore-Penrose invertibility of operators on a Hilbert space
(see \cite[Example~2.16]{HRS01} or \cite[Theorem~4.24]{BS99}). Notice
that the operator $A^L$ (resp., $A^R$) is the Moore-Penrose inverse
of the operator $A$.
%%%----------------------------------------------------------------------------
\subsection{Fundamental property of slowly oscillating functions}
%%%----------------------------------------------------------------------------
\begin{lemma}[{\cite[Proposition~2.2]{K08}}]
\label{le:SO-fundamental-property}
Let $\{a_k\}_{k=1}^\infty$ be a countable subset of $SO(\mR_+)$ and
$s\in\{0,\infty\}$. For each $\xi\in M_s(SO(\mR_+))$ there exists a
sequence $\{t_n\}_{n\in\mN}\subset\mR_+$ such that $t_n\to s$ as $n\to\infty$
and
%%%
\begin{equation}\label{eq:SO-fundamental-property}
a_k(\xi)=\lim_{n\to\infty}a_k(t_n)\quad\mbox{for all}\quad k\in\mN.
\end{equation}
%%%
Conversely, if $\{t_n\}_{n\in\mN}\subset\mR_+$ is a sequence such that
$t_n\to s$ as $n\to\infty$ and the limits $\lim_{n\to\infty}a_k(t_n)$
exist for all $k\in\mN$, then there exists a functional
$\xi\in M_s(SO(\mR_+))$ such that \eqref{eq:SO-fundamental-property} holds.
\end{lemma}
%%%----------------------------------------------------------------------------
\subsection{Properties of iterations of slowly oscillating shifts}
In this subsection we collect some properties of iterations of slowly
oscillating shifts. For $t\in\mR_+$ and $k\in\mN$, let
\[
\alpha_0(t):=t,
\quad
\alpha_k(t):=\alpha[\alpha_{k-1}(t)],
\quad
\alpha_{-k}(t):=\alpha_{-1}[\alpha_{-k+1}(t)].
\]
%%%----------------------------------------------------------------------------
\begin{lemma}[{\cite[Corollary~2.5]{KKL14}}]
\label{le:SOS-iterations}
If $\alpha\in SOS(\mR_+)$, then $\alpha_k\in SOS(\mR_+)$ for every $k\in\mZ$.
\end{lemma}
%%%----------------------------------------------------------------------------
\begin{lemma}[{\cite[Lemma~2.3]{KKL11a}}]
\label{le:composition}
If $c\in SO(\mR_+)$ and $\alpha\in SOS(\mR_+)$, then $c\circ\alpha$ belongs to
$SO(\mR_+)$ and
\[
\lim_{t\to s}(c(t)-c[\alpha(t)])=0
\quad\mbox{for}\quad
s\in\{0,\infty\}.
\]
\end{lemma}
%%%----------------------------------------------------------------------------
\begin{lemma}[{\cite[Lemma~2.6]{KKL16}}]
\label{le:exponent-function-inverse}
Let $\alpha\in SOS(\mR_+)$ and $\alpha_{-1}$ be the inverse function to $\alpha$.
If $\omega$ and $\omega_{-1}$ are the exponent functions of $\alpha$ and
$\alpha_{-1}$, respectively, then $\omega(\xi)=-\omega_{-1}(\xi)$ for all
$\xi\in\Delta$.
\end{lemma}
%%%----------------------------------------------------------------------------
\begin{lemma}\label{le:exponent-function-iterations}
Let $\alpha\in SOS(\mR_+)$ and let $\omega$ be its exponent function.
If $k\in\mZ$ and $\omega_k$ is the exponent function of
$\alpha_k$, then $\omega_k(\xi)=k\omega(\xi)$ for every $\xi\in\Delta$.
\end{lemma}
%%%----------------------------------------------------------------------------
\begin{proof}
For $k=0,1$, the statement is trivial. If $k>1$, then
%%%
\begin{equation}\label{eq:exponent-function-iterations-1}
\omega_k(t)
=
\log\frac{\alpha_k(t)}{t}
=
\log\left(\prod_{j=0}^{k-1}
\frac{\alpha[\alpha_j(t)]}{\alpha_j(t)}\right)
=\sum_{j=0}^{k-1}\omega[\alpha_j(t)],
\quad t\in\mR_+.
\end{equation}
%%%
Since $\omega\in SO(\mR_+)$, we deduce from
Lemmas~\ref{le:SOS-iterations}--\ref{le:composition} that for every
integer
$j\in\{0,\dots,k-1\}$, the function $\omega\circ\alpha_j$ belongs to $SO(\mR_+)$
and
%%%
\begin{equation}\label{eq:exponent-function-iterations-2}
\lim_{t\to s}(\omega(t)-\omega[\alpha_j(t)])=0,
\quad
s\in\{0,\infty\}.
\end{equation}
%%%
Fix $s\in\{0,\infty\}$ and $\xi\in M_s(SO(\mR_+))$. By
Lemma~\ref{le:SO-fundamental-property}, there is a sequence
$\{t_n\}_{n\in\mN}\subset\mR_+$ such that $t_n\to s$ as $n\to\infty$ and
%%%
\begin{equation}\label{eq:exponent-function-iterations-3}
\omega(\xi)=\lim_{n\to\infty}\omega(t_n),
\quad
(\omega\circ\alpha_j)(\xi)=\lim_{n\to\infty}\omega[\alpha_j(t_n)].
\end{equation}
%%%
Equalities \eqref{eq:exponent-function-iterations-2}--\eqref{eq:exponent-function-iterations-3}
imply that for $j\in\{0,\dots,k-1\}$,
\[
(\omega\circ\alpha_j)(\xi)-\omega(\xi)
=
\lim_{n\to\infty}(\omega[\alpha_j(t_n)]-\omega(t_n))
=
0.
\]
We derive from \eqref{eq:exponent-function-iterations-1} and the above
equalities that
\[
\omega_k(\xi)=\sum_{j=0}^{k-1}(\omega\circ\alpha_j)(\xi)=k\omega(\xi),
\]
which completes the proof for $k>1$.

If $k<0$, then we have $\omega_{-k}(\xi)=-k\omega(\xi)$ by the
statement just proved. On the other hand, we deduce from
Lemma~\ref{le:exponent-function-inverse} that 
$\omega_k(\xi)=-\omega_{-k}(\xi)$.
Thus, $\omega_k(\xi)=-(-k)\omega(\xi)=k\omega(\xi)$ for all $\xi\in\Delta$.
\end{proof}
%%%----------------------------------------------------------------------------
\section{Weak one-sided inverse closedness of the algebra
\boldmath{$W_{\alpha,p}^{SO}$}}
\label{sec:Wiener-FO}
%%%----------------------------------------------------------------------------
\subsection{Inverse closedness of the algebra \boldmath{$W_{\alpha,p}^{SO}$ in
the algebra $\cB$}}
Let $\mathfrak{A}\subset\mathfrak{B}$ be two Banach algebras with the same
unit element. Recall that the algebra $\mathfrak{A}$ is said to be inverse
closed in the algebra $\mathfrak{B}$ if for every element $a\in\mathfrak{A}$
invertible in the algebra $\mathfrak{B}$ its inverse $a^{-1}$
belongs to the algebra $\mathfrak{A}$.

We say that the algebra $\mathfrak{A}$ is weakly left (resp.,
right) inverse closed in the algebra $\mathfrak{B}$ if for every element
$a\in\mathfrak{A}$, which is left (resp., right) invertible in the algebra
$\mathfrak{B}$, there exists at least one its left (resp., right) inverse
$a^{(-1)}$ that belongs to the algebra $\mathfrak{A}$.
%%%----------------------------------------------------------------------------
\begin{theorem}[{\cite[Theorem~7.4]{FTK-BJMA}}]
\label{th:inverse-closedness}
For every $p\in(1,\infty)$,
the algebra $W_{p,\alpha}^{SO}$ is inverse closed in the algebra $\cB$.
\end{theorem}
%%%----------------------------------------------------------------------------
\subsection{One-sided inverses belonging to the algebra
\boldmath{$W_{\alpha,p}^\fS$}}
A function $f\in L^\infty(\mR_+)$ is said to be essentially slowly oscillating
(at $0$ and $\infty$) if for each (equivalently, for some) $\lambda\in(0,1)$,
\[
\lim_{r\to s}\operatornamewithlimits{ess\,sup}_{t,\tau\in[\lambda r,r]}
|f(t)-f(\tau)|=0,
\quad s\in\{0,\infty\}.
\]

Fix $\alpha\in SOS(\mR_+)$ and $\tau\in\mR_+$. Consider
the semi-segment $\gamma\subset\mR_+$ with the
endpoints $\tau$ and $\alpha(\tau)$ such that $\tau\in\gamma$ and
$\alpha(\tau)\notin\gamma$. Following \cite[Section~3.2]{FTK-BJMA}, let $\fS$
denote the $C^*$-subalgebra of $L^\infty(\mR_+)$ consisting of all functions
on $\mR_+$ that are continuous on every semi-segment $\alpha_n(\gamma)$ with
$n\in\mZ$, have one-sided limits at the points $\alpha_n(\tau)$ for $n\in\mZ$,
and are essentially slowly oscillating at $0$ and $\infty$.
Let $W_{\alpha,p}^\fS$ be the unital Banach algebra of operators of the form
\eqref{eq:operator-series} with $a_k\in\fS$ for all $k\in\mZ$ and the norm
\[
\|A\|_{W_{\alpha,p}^\fS}=\sum_{k\in\mZ}\|a_k\|_{L^\infty(\mR_+)}<\infty.
\]

From \cite[Theorems~6.3--6.4]{FTK-BJMA} we get the following.
%%%----------------------------------------------------------------------------
\begin{theorem}\label{th:quasi-one-sided-inverse-closedness}
Let $1<p<\infty$, $\alpha\in SOS(\mR_+)$, $a_k\in SO(\mR_+)$ for all $k\in\mZ$,
and
\[
A=\sum_{k\in\mZ}a_k U_\alpha^k\in W_{\alpha,p}^{SO}.
\]
If $A$ is left (resp. right) invertible on $L^p(\mR_+)$, then there exists
a left inverse $A^L$ (resp. right inverse $A^R$) of $A$ such that
$A^L\in W_{\alpha,p}^\fS$ (resp. $A^R\in W_{\alpha,p}^\fS$).
\end{theorem}
%%%----------------------------------------------------------------------------
\subsection{Weak one-sided inverse closedness of the algebra
\boldmath{$W_{\alpha,p}^{SO}$ in $\cB$}}
We will show that the algebra $W_{\alpha,p}^{SO}$ is weakly left and right
inverse closed in the algebra $\cB$. For every operator
$A\in W_{\alpha,p}^{SO}$ of the form \eqref{eq:operator-series},
define its formally adjoint $A^\diamond$ by
\[
A^\diamond:=
\sum_{k\in\mZ}(\overline{a_k}\circ\alpha_{-k})U_\alpha^{-k}
\in W_{\alpha,p}^{SO}.
\]
%%%----------------------------------------------------------------------------
\begin{theorem}\label{th:one-sided-inverse-closedness}
Let $1<p<\infty$, $\alpha\in SOS(\mR_+)$, $a_k\in SO(\mR_+)$ for all $k\in\mZ$,
and
\[
A=\sum_{k\in\mZ}a_k U_\alpha^k\in W_{\alpha,p}^{SO}.
\]
\begin{enumerate}
\item[(a)]
If $A$ is left invertible on $L^p(\mR_+)$, then the operator
$A^\diamond A$ is invertible on the space $L^p(\mR_+)$, the
operator $A^L:=(A^\diamond A)^{-1}A^\diamond$
is a left inverse of $A$, and $A^L\in W_{\alpha,p}^{SO}$.

\item[(b)]
If $A$ is right invertible on $L^p(\mR_+)$, then the operator
$AA^\diamond$ is invertible on the space $L^p(\mR_+)$, the
operator $A^R:=A^\diamond(AA^\diamond)^{-1}$
is a right inverse of $A$, and $A^R\in W_{\alpha,p}^{SO}$.
\end{enumerate}
\end{theorem}
%%%----------------------------------------------------------------------------
\begin{proof}
Along with the operator $U_\alpha$ acting on $L^p(\mR_+)$, consider the
operator $U_{\alpha,2}$ acting on $L^2(\mR_+)$ and defined by the same rule
$U_{\alpha,2}f=(\alpha')^{1/2}(f\circ\alpha)$. Then we can define the canonical
isometric isomorphisms of Banach algebras
\[
\Psi_\fS:W_{\alpha,p}^\fS\to W_{\alpha,2}^\fS,
\quad
\Psi_{SO}:W_{\alpha,p}^{SO}\to W_{\alpha,2}^{SO}
\]
by the formulas
\[
\Psi_\fS\left(\sum_{k\in\mZ}c_k U_\alpha^k\right)=\sum_{k\in\mZ}c_kU_{\alpha,2}^k,
\quad
\Psi_{SO}\left(\sum_{k\in\mZ}c_k U_\alpha^k\right)=\sum_{k\in\mZ}c_kU_{\alpha,2}^k,
\]
respectively.

If $A\in W_{\alpha,p}^{SO}$ is left invertible on the space
$L^p(\mR_+)$, then by Theorem~\ref{th:quasi-one-sided-inverse-closedness},
there exists an operator $\widetilde{A}^L\in W_{\alpha,p}^\fS$ such that
$\widetilde{A}^LA=I$ on $L^p(\mR_+)$. Hence
$\Psi_\fS(\widetilde{A}^L)\Psi_\fS(A)=I$ on $L^2(\mR_+)$. Therefore, the operator
\[
A_2:=\Psi_{SO}(A)=\Psi_\fS(A)\in W_{\alpha,2}^{SO}
\]
is left invertible on $L^2(\mR_+)$. Hence, in view of
Lemma~\ref{le:Hilbert-one-sided}, the operator $A_2^*A_2$ is invertible on
$L^2(\mR_+)$. Observe that $A_2^*\in W_{\alpha,2}^{SO}$ and
$A^\diamond=\Phi_{SO}^{-1}(A_2^*)\in W_{\alpha,p}^{SO}$.
Since $A_2^*A_2\in W_{\alpha,2}^{SO}$, we deduce from the inverse closedness
of the algebra $W_{\alpha,2}^{SO}$ in the algebra $\cB$ (see
Theorem~\ref{th:inverse-closedness}) that $(A_2^*A_2)^{-1}\in W_{\alpha,2}^{SO}$.
Then  $\Psi_{SO}^{-1}((A_2^*A_2)^{-1})\in W_{\alpha,p}^{SO}$. Now it is easy
to check that
\[
(A^\diamond A)^{-1}=\Psi_{SO}^{-1}((A_2^*A_2)^{-1}).
\]
Hence $A^L=(A^\diamond A)^{-1}A$ is a left inverse to $A$. Part (a) is proved.

(b) The proof of part (b) is reduced to the previous one by passing
to adjoint operators.
\end{proof}
%%%----------------------------------------------------------------------------
\section{Algebra \boldmath{$\cA$} of singular integral operators}
\label{sec:algebra-A}
\subsection{Fourier and Mellin convolution operators}
Let $\cF:L^2(\mR)\to L^2(\mR)$ denote the Fourier transform,
\[
(\cF f)(x):=\int_\mR f(y)e^{-ixy}dy,\quad x\in\mR,
\]
and let $\cF^{-1}:L^2(\mR)\to L^2(\mR)$ be the inverse of $\cF$. A function
$a\in L^\infty(\mR)$ is called a Fourier multiplier on $L^p(\mR)$ if the
mapping $f\mapsto \cF^{-1}a\cF f$ maps $L^2(\mR)\cap L^p(\mR)$ into itself and
extends to a bounded operator on $L^p(\mR)$. The latter operator is then
denoted by $W^0(a)$. We let $\cM_p(\mR)$ stand for the set of all Fourier
multipliers on $L^p(\mR)$. One can show that $\cM_p(\mR)$ is a Banach algebra
under the norm
\[
\|a\|_{\cM_p(\mR)}:=\|W^0(a)\|_{\cB(L^p(\mR))}.
\]

Let $d\mu(t)=dt/t$ be the (normalized) invariant measure on $\mR_+$.
Consider the Fourier transform on $L^2(\mR_+,d\mu)$, which is
usually referred to as the Mellin transform and is defined by
\[
\cM:L^2(\mR_+,d\mu)\to L^2(\mR),
\quad
(\cM f)(x):=\int_{\mR_+} f(t) t^{-ix}\,\frac{dt}{t}.
\]
This operator is invertible, with inverse given by
\[
{\cM^{-1}}:L^2(\mR)\to L^2(\mR_{+},d\mu),
\quad
({\cM^{-1}}g)(t)= \frac{1}{2\pi}\int_{\mR}
g(x)t^{ix}\,dx.
\]
Let $E$ be the isometric isomorphism
%%%
\begin{equation}\label{eq:def-E}
E:L^p(\mR_+,d\mu)\to L^p(\mR),
\quad
(Ef)(x):=f(e^x),\quad x\in\mR.
\end{equation}
%%%
Then the map
%%%
\begin{equation}\label{eq:R-to-R+}
A\mapsto E^{-1}AE
\end{equation}
%%%
transforms the Fourier convolution operator $W^0(a)=\cF^{-1}a\cF$ to the
Mellin convolution operator
\[
\Co(a):=\cM^{-1}a\cM
\]
with the same symbol $a$. Hence the class of Fourier multipliers on
$L^p(\mR)$ coincides with the class of Mellin multipliers on $L^p(\mR_+,d\mu)$.
%%%-----------------------------------------------------------------------
\subsection{Continuous and piecewise continuous multipliers}
We denote by $PC$ the $C^*$-algebra of all bounded piecewise continuous
functions on $\dot{\mR}=\mR\cup\{\infty\}$. By definition, $a\in PC$ if
and only if $a\in L^\infty(\mR)$ and the one-sided limits
\[
a(x_0-0):=\lim_{x\to x_0-0}a(x),
\quad
a(x_0+0):=\lim_{x\to x_0+0}a(x)
\]
exist for each $x_0\in\dot{\mR}$. If a function $a$ is given everywhere
on $\mR$, then its total variation is defined by
\[
V(a):=\sup\sum_{k=1}^n|a(x_k)-a(x_{k-1})|,
\]
where the supremum is taken over all $n\in\mN$ and
\[
-\infty<x_0<x_1<\dots<x_n<+\infty.
\]
If $a$ has a finite total variation, then it has finite one-sided limits
$a(x-0)$ and $a(x+0)$ for all $x\in\dot{\mR}$, that is, $a\in PC$
(see, e.g., \cite[Chap. VIII, Sections~3 and~9]{N55}). The following
theorem gives an important subset of $\cM_p(\mR)$. Its proof can be found,
e.g., in \cite[Theorem~17.1]{BKS02} or \cite[Theorem~2.11]{D79}.
%%%----------------------------------------------------------------------------
\begin{theorem}[Stechkin's inequality]
\label{th:Stechkin}
If $a\in PC$ has finite total variation $V(a)$, then $a\in\cM_p(\mR)$ and
\[
\|a\|_{\cM_p(\mR)}\le\|S_\mR\|_{\cB(L^p(\mR))}\big(\|a\|_{L^\infty(\mR)}+V(a)\big),
\]
where $S_\mR$ is the Cauchy singular integral operator on $\mR$.
\end{theorem}
%%%----------------------------------------------------------------------------
According to \cite{D79} or \cite[p.~325]{BKS02}, let $PC_p$ be the closure in
$\cM_p(\mR)$ of the set of all functions $a\in PC$ with finite total variation
on $\mR$. Following \cite[p.~331]{BKS02}, put
\[
C_p(\overline{\mR}):=PC_p\cap C(\mR),
\]
where $\overline{\mR}:=[-\infty,+\infty]$. It is easy to see that $PC_p$ and
$C_p(\overline{\mR})$ are Banach algebras.
%%%----------------------------------------------------------------------------
\subsection{Maximal ideal space of the algebras \boldmath{$\cA$ and $\cA^\pi$}}
Let $\mathfrak{A}$ be a Banach algebra and $\mathfrak{C}$ be a subset of
$\mathfrak{A}$. Following \cite[Section~3.45]{BS06}, we denote  by
$\alg_\mathfrak{A}\mathfrak{C}$ the smallest closed subalgebra of
$\mathfrak{A}$ containing $\mathfrak{C}$
and by $\id_\mathfrak{A}\mathfrak{C}$ the smallest closed
two-sided ideal of $\mathfrak{A}$ containing $\mathfrak{C}$.

Put $\cA:=\alg_\cB\{I,S_0\}$. Obviously, the algebra $\cA$ is commutative.
Consider the isometric isomorphism
\[
\Phi:L^p(\mR_+)\to L^p(\mR_+,d\mu),
\quad
(\Phi f)(t):=t^{1/p}f(t),
\quad t\in\mR_+.
\]

The following result is well known. It is essentially due to Duduchava
\cite{D79,D87} and Simonenko, Chin Ngok Minh \cite{SCNM86}. It can be found
with a proof in \cite[Section~1.10.2]{DS08}, \cite[Section~2.1.2]{HRS94},
\cite[Sections~4.2.2-4.2.3]{RSS11}.
%%%----------------------------------------------------------------------------
\begin{theorem}\label{th:algebra-A}
\begin{enumerate}
\item[{\rm (a)}]
The algebra $\cA$ is the smallest closed subalgebra of $\cB$ that contains
the operators $\Phi^{-1}\Co(a)\Phi$ with $a\in C_p(\overline{\mR})$.

\item[{\rm (b)}]
The maximal ideal space of the commutative Banach algebra $\cA$ is homeomorphic
to $\overline{\mR}$. In particular, the operator $\Phi^{-1}\Co(a)\Phi$
with $a\in C_p(\overline{\mR})$ is invertible if and only if $a(x)\ne 0$ for
all $x\in\overline{\mR}$. Thus $\cA$ is an inverse closed subalgebra of $\cB$.

\item[{\rm (c)}]
The operator $\Phi^{-1}\Co(a)\Phi$ with $a\in C_p(\overline{\mR})$
belongs to $\id_\cA\{R_0\}$ if and only if $a(-\infty)=a(+\infty)=0$.

\item[{\rm (d)}]
If $\gamma\in\mC$ satisfies \eqref{eq:gamma-condition}, then the function
$s_\gamma$ given by \eqref{eq:def-p-gamma-pm} and the function $r_\gamma$
defined by
\[
r_\gamma(x):=1/\sinh[\pi(x+i/p+i\gamma)], \quad x\in\mR,
\]
belong to $C_p(\overline{\mR})$ and
\[
S_\gamma=\Phi^{-1}\Co(s_\gamma)\Phi,
\quad
R_\gamma=\Phi^{-1}\Co(r_\gamma)\Phi.
\]
\end{enumerate}
\end{theorem}
%%%----------------------------------------------------------------------------
Let us describe the quotient algebra
\[
\cA^\pi=(\cA+\cK)/\cK.
\]
By \cite[Proposition~4.2.14]{RSS11}, a Mellin convolution operator is Fredholm
on the space $L^p(\mR_+,d\mu)$ if and only if it is invertible on this space.
Hence, Theorem~\ref{th:algebra-A} implies the following.
%%%----------------------------------------------------------------------------
\begin{corollary}\label{co:algebra-Api}
\begin{enumerate}
\item[{\rm (a)}]
The algebra $\cA^\pi$ is commutative and its maximal ideal space is
homeomorphic to $\overline{\mR}$.

\item[{\rm (b)}]
The Gelfand transform of the coset $(\Phi^{-1}\Co(a)\Phi)^\pi\in\cA^\pi$
for $a\in C_p(\overline{\mR})$ is given by
\[
\big[(\Phi^{-1}\Co(a)\Phi)^\pi\big]
\widehat{\hspace{2mm}}(x)=a(x)\quad\mbox{for}\quad x\in\overline{\mR}.
\]
\end{enumerate}
\end{corollary}
%%%----------------------------------------------------------------------------
\subsection{Some operator relations}
%%%----------------------------------------------------------------------------
\begin{lemma}
[{\cite[Lemma~2.4]{K15}, \cite[Lemma~4.2]{KKL-JIEA}}]
\label{le:PR-relations}
Let $1<p<\infty$ and $\gamma,\delta\in\mC$ be such that $0<1/p+\Re\gamma<1$
and $0<1/p+\Re\delta<1$. Then
\[
P_\delta^+-P_\gamma^+
=
P_\gamma^--P_\delta^-
=
\frac{1}{2}\sinh[\pi i(\gamma-\delta)]R_\gamma R_\delta,
\quad
P_\gamma^- P_\delta^+
=
-\frac{e^{i\pi(\delta-\gamma)}}{4}R_\gamma R_\delta.
\]
\end{lemma}
%%%---------------------------------------------------------------------------
\subsection{Compactness of commutators of singular integral operators and\\
functional operators}
Fix $\alpha,\beta\in SOS(\mR_+)$ and consider the Banach algebra of functional
operators with shifts and slowly oscillating data defined by
\[
\mathcal{FO}_{\alpha,\beta}:=
\alg_\cB\{U_\alpha,U_\alpha^{-1},U_\beta,U_\beta^{-1},cI\ :\ c\in SO(\mR_+)\}.
\]
%%%---------------------------------------------------------------------------
\begin{lemma}[{\cite[Lemma~2.8]{KKL16}}]
\label{le:compactness-commutators}
Let $\alpha,\beta\in SOS(\mR_+)$. If $A\in\mathcal{FO}_{\alpha,\beta}$ and
$B\in\cA$, then $AB\simeq BA$.
\end{lemma}
%%%----------------------------------------------------------------------------
\section{Allan-Douglas localization}
\label{sec:localization}
\subsection{The Allan-Douglas local principle}
Let $\fA$ be a Banach algebra with identity. A subalgebra $\fZ$ of $\fA$ is
said to be a central subalgebra of $\fA$ if $za=az$ for all $z\in\fZ$ and all
$a\in\fA$.

The proof of the following result is contained,  e.g., in
\cite[Theorem~1.35(a)]{BS06}.
%%%----------------------------------------------------------------------------
\begin{theorem}[Allan-Douglas]\label{th:Allan-Douglas}
Let $\fA$ be a Banach algebra with identity $e$ and let $\fZ$
be a closed central subalgebra of $\fA$ containing
$e$. Let $M(\fZ)$ be the maximal ideal space of $\fZ$, and for
$\omega\in M(\fZ)$, let $\fJ_\omega$ refer to the smallest closed
two-sided ideal of $\fA$ containing the ideal $\omega$. Then an
element $a$ is left (resp., right, two-sided) invertible in $\fA$
if and only if $a+\fJ_\omega$ is left (resp., right, two-sided)
invertible in the quotient algebra $\fA/\fJ_\omega$ for all
$\omega\in M(\fZ)$.
\end{theorem}
%%%----------------------------------------------------------------------------
The algebra $\fA/\fJ_\omega$ is referred to as the local
algebra of $\fA$ at $\omega\in M(\fZ)$.
%%%----------------------------------------------------------------------------
\subsection{Algebras of singular integral operators
with shifts and algebras\\ of operators of local type}
Following \cite[Section~6.3]{KKL11a}, we consider the following sets:
%%%
\begin{align*}
\cZ
&:=
\alg_{\cB}\{I,S_0,cR_0,K\ :\ c\in SO(\mR_+),\ K\in\cK\},
\\
\Lambda
&:=
\{A\in\cB\ :\ AC-CA\in\cK {\rm \ \ for\ all\ \ } C\in\cZ\}.
\end{align*}
%%%
By \cite[Lemma~6.7(a)]{KKL11a}, the set $\Lambda$ is a closed unital subalgebra
of the algebra $\cB$, which is usually called the algebra of operators of
local type.

For $\alpha,\beta\in SOS(\mR_+)$, put
\[
\cF_{\alpha,\beta}
:=
\alg_{\cB} (\{S_0\}\cup\mathcal{FO}_{\alpha,\beta}).
\]

By a minor modification of the proof of \cite[Theorem~6.8]{KKL11a}
with the aid of Lemma~\ref{le:compactness-commutators}, we get the following.
%%%----------------------------------------------------------------------------
\begin{theorem}\label{th:embeddings}
We have $\cK\subset\cZ\subset\cF_{\alpha,\beta}\subset\Lambda$.
\end{theorem}
%%%----------------------------------------------------------------------------
\subsection{Maximal ideal space of the algebra \boldmath{$\cZ^\pi$}}
It follows from Theorem~\ref{th:embeddings} that the quotient algebras
$\cZ^\pi:=\cZ/\cK$ and $\Lambda^\pi:=\Lambda/\cK$ are well defined.
Clearly, $\cZ^\pi$ lies in the center of $\Lambda^\pi$.
%%%----------------------------------------------------------------------------
\begin{theorem}[{\cite[Theorem~6.11]{KKL11a}}]
\label{th:algebra-Zpi}
For the commutative Banach algebra $\cZ^\pi$ the following statements hold:
\begin{enumerate}
\item[{\rm (a)}]
the maximal ideal space $M(\cZ^\pi)$ of $\cZ^\pi$ is homeomorphic to
the set
\[
\{-\infty,+\infty\}\cup(\Delta\times\mR);
\]

\item[{\rm (b)}]
any coset $Z^\pi\in\cZ^\pi$ is of the form
%%%
\begin{equation}\label{eq:coset-of-Zpi}
Z^\pi=(c_+P_0^+)^\pi+(c_-P_0^-)^\pi+\lim_{n\to\infty}\sum_{k=1}^{m_n}(c_{n,k}H_{n,k})^\pi
\end{equation}
%%%
where $c_\pm\in\mC$, $c_{n,k}\in SO(\mR_+)$, $H_{n,k}\in\id_\cA\{R_0\}$, and $m_n\in\mN$;

\item[{\rm (c)}]
the Gelfand transform of the coset $Z^\pi\in\cZ^\pi$ defined by
\eqref{eq:coset-of-Zpi} is given for a point $(\xi,x)\in\Delta\times\mR$ by
\[
(Z^\pi)\widehat{\hspace{2mm}}(\xi,x)=c_+p_0^+(x)+c_-p_0^-(x)+
\lim_{n\to\infty}\sum_{k=1}^{m_n}c_{n,k}(\xi)(H_{n,k}^\pi)\widehat{\hspace{2mm}}(x),
\]
where $(H_{n,k}^\pi)\widehat{\hspace{2mm}}(x)$ is the Gelfand transform of
a coset $H_{n,k}^\pi\in\cA^\pi$, which is calculated in 
Corollary~$\ref{co:algebra-Api}({\rm b})$.
\end{enumerate}
\end{theorem}
%%%----------------------------------------------------------------------------
\subsection{Semi-Fredholmness of operators of local type}
Let $\cJ_{+\infty}^\pi$, $\cJ_{-\infty}^\pi$ and $\cJ_{\xi,x}^\pi$ for
$(\xi,x)\in\Delta\times\mR$ be the closed two-sided ideals of the Banach
algebra $\Lambda^\pi$ generated, respectively, by the maximal ideals
%%%
\begin{align*}
\cI_{+\infty}^\pi
&:=
\id_{\cZ^\pi}
\big\{(P_0^-)^\pi,(gR_0)^\pi :\ g\in SO(\mR_+)\big\},
\\
\cI_{-\infty}^\pi
&:=
\id_{\cZ^\pi}
\big\{(P_0^+)^\pi,(gR_0)^\pi :\ g\in SO(\mR_+)\big\},
\\
\cI_{\xi,x}^\pi
&:=
\big\{Z^\pi\in\cZ^\pi:
(Z^\pi)\widehat{\hspace{2mm}}(\xi,x)=0\big\}
\end{align*}
%%%
of the algebra $\cZ^\pi$, and let
\[
\Lambda^\pi_{+\infty}:=\Lambda^\pi/\cJ^\pi_{+\infty},
\quad
\Lambda^\pi_{-\infty}:=\Lambda^\pi/\cJ^\pi_{-\infty},
\quad
\Lambda^\pi_{\xi,x}:=\Lambda^\pi/\cJ^\pi_{\xi,x}
\]
be the corresponding quotient algebras (see also \cite[Section~6.6]{KKL11a}).

Obviously, an operator $T\in\Lambda$ is left Fredholm
(resp., right Fredholm) on the space $L^p(\mR_+)$ if the
coset $T^\pi:=A+\cK$ is left invertible (resp., right
invertible) in the quotient Banach algebra $\Lambda^\pi\subset\cB^\pi$.
Applying now Theorem~\ref{th:Allan-Douglas} with $\fA=\Lambda^\pi$ and
$\fZ=\cZ^\pi$, we immediately obtain the following.
%%%----------------------------------------------------------------------------
\begin{theorem}\label{th:localization-left-right}
An operator $T\in\Lambda$ is left (resp., right) Fredholm
on the space $L^p(\mR_+)$ if the following three conditions are fulfilled:
\begin{itemize}
\item[{\rm (i)}]
the coset $T^\pi+\cJ_{+\infty}^\pi$ is left (resp., right) invertible
in the quotient algebra ${\Lambda}_{+\infty}^\pi$;

\item[{\rm (ii)}]
the coset $T^\pi+\cJ_{-\infty}^\pi$ is left (resp., right) invertible
in the quotient algebra ${\Lambda}_{-\infty}^\pi$;

\item[{\rm (iii)}]
for every $(\xi,x)\in\Delta\times\mR$, the coset $T^\pi+\cJ_{\xi,x}^\pi$
is left (resp., right) invertible  in the quotient algebra
${\Lambda}_{\xi,x}^\pi$.
\end{itemize}
\end{theorem}
%%%----------------------------------------------------------------------------
It follows from Theorems~\ref{th:algebra-A}(d) and \ref{th:embeddings} that
$N\in\cF_{\alpha,\beta}\subset\Lambda$. Thus,
Theorem~\ref{th:localization-left-right} is applicable to
$N$. Hence, our next aim is to study one-sided invertibility of the cosets
$N^\pi+\cJ^\pi_{+\infty}$, $N^\pi+\cJ^\pi_{-\infty}$ and
$N^\pi+\cJ^\pi_{\xi,x}$ in the corresponding local algebras
$\Lambda^\pi_{+\infty}$, $\Lambda^\pi_{-\infty}$ and $\Lambda^\pi_{\xi,x}$
for all $(\xi,x)\in\Delta\times\mR$.
%%%----------------------------------------------------------------------------
\section{Mellin pseudodifferential operators and their symbols}
\label{sec:Mellin-PDO}
%%%----------------------------------------------------------------------------
\subsection{Mellin PDO's: overview}
Mellin pseudodifferential operators are generalizations of Mellin convolution
operators. Let $\fa$ be a sufficiently smooth function defined on
$\mR_+\times\mR$. The Mellin pseudodifferential operator (shortly, Mellin PDO)
with symbol $\fa$ is initially defined for smooth functions $f$ of compact
support by the iterated integral
%%%
\begin{align}
[\Op(\fa) f](t)
&=
[\cM^{-1}\fa(t,\cdot)\cM f](t)
\nonumber\\
&=
\frac{1}{2\pi}\int_\mR dx \int_{\mR_+}
\fa(t,x)\left(\frac{t}{\tau}\right)^{ix}f(\tau) \frac{d\tau}{\tau}
\quad\mbox{for}\quad t\in\mR_+.
\label{eq:Mellin-PDO}
\end{align}
%%%
Obviously, if $\fa(t,x)=a(x)$ for all $(t,x)\in\mR_+\times\mR$, then the Mellin
pseudodifferential operator $\Op(\fa)$ becomes the Mellin convolution operator
\[
\Op(\fa)=\Co(a).
\]

In 1991 Rabinovich \cite{R92} (see also \cite{R98}) proposed to use Mellin
pseudodifferential operators with $C^\infty$ slowly oscillating symbols to
study singular integral operators with slowly oscillating coefficients
on $L^p$ spaces. Namely, he considered symbols
$\fa\in C^\infty(\mR_+\times\mR)$ such that
%%%
\begin{equation}\label{eq:Hoermander}
\sup_{(t,x)\in\mR_+\times\mR}
\big|(t\partial_t)^j\partial_x^k\fa(t,x)\big|(1+x^2)^{k/2}<\infty
\quad\mbox{for all}\quad j,k\in\mZ_+
\end{equation}
%%%
and
%%%
\begin{equation}\label{eq:Grushin}
\lim_{t\to s}\sup_{x\in\mR}
\big|(t\partial_t)^j\partial_x^k\fa(t,x)\big|(1+x^2)^{k/2}=0
\mbox{ for all }\ j\in\mN,\quad k\in\mZ_+,
\end{equation}
%%%
where $\mZ_+=\mN\cup\{0\}$ and $s\in\{0,\infty\}$. Here and in what follows 
$\partial_t$ and $\partial_x$ denote the operators of partial differentiation 
with respect to $t$ and to $x$. Notice that \eqref{eq:Hoermander} defines
nothing but the Mellin version of the H\"ormander class $S_{1,0}^0(\mR)$ 
(see, e.g., \cite[Chap.~2, Section~1]{K82} for the definition of the  
H\"ormander classes $S_{\varrho,\delta}^m(\mR^n)$). If $\fa$ satisfies 
\eqref{eq:Hoermander}, then the Mellin PDO $\Op(\fa)$ is bounded on the 
spaces $L^p(\mR_+,d\mu)$ for $1<p<\infty$ (see, e.g.,
\cite[Chap.~VI, Proposition~4]{St93} for the corresponding Fourier PDO's).
Condition \eqref{eq:Grushin} is the Mellin version of Grushin's definition of
symbols slowly varying in the first variable (see, e.g., \cite{G70},
\cite[Chap.~3, Definition~5.11]{K82}).

The idea of application of Mellin PDO's with considered class of symbols
was exploited in a series of papers by Rabinovich and coauthors (see
\cite[Sections~4.6--4.7]{RRS04} for the complete history up to 2004).
On the other hand, the smoothness conditions imposed on slowly oscillating
symbols are very strong. In particular, they are not applicable directly
to the problem we are dealing with in the present paper.

In 2005 the second author \cite{K06} developed a Fredholm theory for the
Fourier pseudodifferential operators with slowly oscillating symbols of
limited smoothness in the spirit of Sarason's definition \cite[p.~820]{S77}
of slow oscillation adopted in the present paper (much less
restrictive than in \cite{R98} and in the works mentioned in \cite{RRS04}).
Necessary for our purposes results from \cite{K06} were translated to
the Mellin setting, for instance, in \cite{KKL14} with the aid of the
transformation defined by \eqref{eq:def-E}--\eqref{eq:R-to-R+}. For the
convenience of readers, we reproduce the results required in what follows
exactly in the same form as they were stated in \cite{KKL14}, where more
details on their proofs can be found.
%%%----------------------------------------------------------------------------
\subsection{Boundedness of Mellin PDO's}
If $a$ is an absolutely continuous function of finite total variation on $\mR$,
then its derivative belongs to $L^1(\mR)$ and
\[
V(a)=\int_\mR|a'(x)|\,dx
\]
(see, e.g., \cite[Chap. VIII, Sections 3 and 9; Chap. XI, Section~4]{N55}).
The set $V(\mR)$ of all absolutely continuous functions of finite total
variation on $\mR$ forms a Banach algebra when equipped with the norm
%%%
\begin{equation}\label{eq:V-norm}
\|a\|_V:=\|a\|_{L^\infty(\mR)}+V(a)=\|a\|_{L^\infty(\mR)}+\int_\mR|a'(x)|\,dx.
\end{equation}

Let $C_b(\mR_+,V(\mR))$ denote the Banach algebra of all bounded continuous
$V(\mR)$-valued functions on $\mR_+$ with the norm
\[
\|\fa(\cdot,\cdot)\|_{C_b(\mR_+,V(\mR))}
=
\sup_{t\in\mR_+}\|\fa(t,\cdot)\|_V.
\]
As usual, let $C_0^\infty(\mR_+)$ be the set of all infinitely differentiable
functions of compact support on $\mR_+$.
%%%----------------------------------------------------------------------------
\begin{theorem}[{\cite[Theorem~3.1]{KKL14}}]
\label{th:boundedness-PDO}
If $\fa\in C_b(\mR_+,V(\mR))$, then the Mellin pseudodifferential operator
$\Op(\fa)$, defined for functions $f\in C_0^\infty(\mR_+)$ by the iterated
integral \eqref{eq:Mellin-PDO}, extends to a bounded linear operator on
the space $L^p(\mR_+,d\mu)$ and there is a number $C_p\in(0,\infty)$ depending
only on $p$ such that
\[
\|\Op(\fa)\|_{\cB(L^p(\mR_+,d\mu))}\le C_p\|\fa\|_{C_b(\mR_+,V(\mR))}.
\]
\end{theorem}
%%%----------------------------------------------------------------------------
\subsection{Products of Mellin PDO's}
Consider the Banach subalgebra $SO(\mR_+,V(\mR))$ of the algebra
$C_b(\mR_+,V(\mR))$ consisting of all $V(\mR)$-valued functions $\fa$ on
$\mR_+$ that slowly oscillate at $0$ and $\infty$, that is,
\[
\lim_{r\to s} \max_{t,\tau\in[r,2r]}\|\fa(t,\cdot)-\fa(\tau,\cdot)\|_{L^\infty(\mR)}=0,
\quad s\in\{0,\infty\}.
\]

Let $\cE(\mR_+,V(\mR))$ be the Banach algebra of all $V(\mR)$-valued functions
$\fa$ in the algebra $SO(\mR_+,V(\mR))$ such that
\[
\lim_{|h|\to 0}\sup_{t\in\mR_+}\big\|\fa(t,\cdot)-\fa^h(t,\cdot)\big\|_V=0
\]
where $\fa^h(t,x):=\fa(t,x+h)$ for all $(t,x)\in\mR_+\times\mR$.
%%%---------------------------------------------------------------------------
\begin{theorem}[{\cite[Theorem~3.3]{KKL14}}]
\label{th:comp-semi-commutators-PDO}
If $\fa,\fb\in\cE(\mR_+,V(\mR))$, then
\[
\Op(\fa)\Op(\fb)\simeq \Op(\fa\fb).
\]
\end{theorem}
%%%---------------------------------------------------------------------------
\begin{lemma}[{\cite[Lemma~3.4]{KKL14}}]
\label{le:PDO-3-operators}
If $\fa,\fb,\fc\in\cE(\mR_+,V(\mR))$ are such that $\fa$ depends only on the
first variable and $\fc$ depends only on the second variable, then
\[
\Op(\fa)\Op(\fb)\Op(\fc)
=
\Op(\fa\fb\fc).
\]
\end{lemma}
%%%----------------------------------------------------------------------------
\subsection{Applications of Mellin pseudodifferential operators}
We immediately deduce the following assertion from \cite[Lemma~4.1]{K15}.
%%%----------------------------------------------------------------------------
\begin{lemma}
\label{le:ff-fs-fr}
Suppose $f\in SO(\mR_+)$ and $\gamma\in\mC$ satisfies \eqref{eq:gamma-condition}.
Then the functions
\[
\ff(t,x):=f(t),
\quad
\fs_\gamma(t,x):=s_\gamma(x),
\quad
\fr_\gamma(t,x):=r_\gamma(x),
\quad
(t,x)\in\mR_+\times\mR,
\]
belong to the Banach algebra $\cE(\mR_+,V(\mR))$.
\end{lemma}
%----------------------------------------------------------------------------
\begin{lemma}\label{le:PDO-in-Lambda}
If $\fb\in\cE(\mR_+,V(\mR))$, then the operator $\Phi^{-1}\Op(\fb)\Phi$
belongs to the algebra $\Lambda$.
\end{lemma}
%------------------------------------------------------------------------------
\begin{proof}
Let $c\in SO(\mR_+)$. It follows from Lemma~\ref{le:ff-fs-fr} that the
functions
\[
\fc_0(t,x):=c(t)r_0(x),
\quad
\fs_0(t,x):=s_0(x),
\quad
(t,x)\in\mR_+\times\mR,
\]
belong to the algebra $\cE(\mR_+,V(\mR))$. Since
$\cE(\mR_+,V(\mR))\subset C_b(\mR_+,V(\mR))$,
Theorem~\ref{th:boundedness-PDO} implies that $B:=\Phi^{-1}\Op(\fb)\Phi\in\cB$.
We infer from Theorem~\ref{th:comp-semi-commutators-PDO} that
%%%
\begin{align}
\Op(\fs_0)\Op(\fb)-\Op(\fb)\Op(\fs_0)\in\cK(L^p(\mR_+,d\mu)),
\label{eq:PDO-in-Lambda-1}
\\
\Op(\fc_0)\Op(\fb)-\Op(\fb)\Op(\fc_0)\in\cK(L^p(\mR_+,d\mu)).
\label{eq:PDO-in-Lambda-2}
\end{align}
%%%
On the other hand, by Theorem~\ref{th:algebra-A}(d) and
Lemma~\ref{le:PDO-3-operators},
%%%
\begin{align}
&
S_0=\Phi^{-1}\Co(s_0)\Phi=\Phi^{-1}\Op(\fs_0)\Phi,
\label{eq:PDO-in-Lambda-3}
\\
&
cR_0=c\Phi^{-1}\Co(r_0)\Phi=\Phi^{-1}\Op(\fc_0)\Phi.
\label{eq:PDO-in-Lambda-4}
\end{align}
%%%
Combining \eqref{eq:PDO-in-Lambda-1}--\eqref{eq:PDO-in-Lambda-4}, we
conclude that $S_0B-BS_0,(cR_0)B-B(cR_0)\in\cK$. Hence $B\in\Lambda$.
\end{proof}
%%%----------------------------------------------------------------------------
Applying \cite[Lemma~4.4]{K15} and making minor modifications in the proof of
\cite[Lemma~4.5]{KKL14}, we get the following.
%%%----------------------------------------------------------------------------
\begin{lemma}\label{le:shift-R-gamma}
Let $\gamma\in\mC$ satisfy \eqref{eq:gamma-condition}. Suppose
$\alpha\in SOS(\mR_+)$, $\omega$ is its exponent function, and $U_\alpha$ is
the associated isometric shift operator on $L^p(\mR_+)$. Then the operator
$U_\alpha R_\gamma$ can be realized as the Mellin pseudodifferential
operator up to a compact operator:
\[
U_\alpha R_\gamma \simeq \Phi^{-1}\Op (\fc) \Phi,
\]
where the function $\fc$, given by
\[
\fc(t,x):=e^{i\omega(t)x}r_\gamma(x)\ \ \mbox{ for }\ \
(t,x)\in\mR_+\times\mR,
\]
belongs to the Banach algebra $\cE(\mR_+,V(\mR))$.
\end{lemma}
%%%----------------------------------------------------------------------------
\section{Sufficient conditions for the semi-Fredholmness}
\label{sec:proofs}
\subsection{One-sided invertibility in the quotient algebras
\boldmath{$\Lambda_{+\infty}^\pi$} and \boldmath{$\Lambda_{-\infty}^\pi$}}
%%%----------------------------------------------------------------------------
The next lemma shows that the operator $N$ can be written as a paired operator
with respect to the pair $(P_0^+,P_0^-)$.
%%%----------------------------------------------------------------------------
\begin{lemma}\label{le:N-representations}
Let $1<p<\infty$ and let $\gamma\in\mC$ satisfy \eqref{eq:gamma-condition}.
Suppose $a_k,b_k$ belong to $SO(\mR_+)$ for all $k\in\mZ$,
$\alpha,\beta$ belong to $SOS(\mR_+)$, the operators
$A_+\in W_{\alpha,p}^{SO}$ and $A_-\in W_{\beta,p}^{SO}$
are given by \eqref{eq:def-Apm}, and the operator $N$ is given by
\eqref{eq:def-N}. Then the operator $N$ can be represented in each of the
forms
\begin{equation}\label{eq:suf-i-1}
N=A_+P_0^++C_-P_0^-=C_+P_0^++A_-P_0^-,
\end{equation}
%%%
where
\begin{align*}
C_+ &:=A_++2\sinh(\pi i\gamma)e^{\pi i\gamma}(A_+-A_-)P_\gamma^-,
\\
C_-&:=A_-+2\sinh(\pi i\gamma)e^{-\pi i \gamma}(A_+-A_-)P_\gamma^+.
\end{align*}
%%%
Moreover, all operators $DP_0^+-P_0^+ D$ and $DP_0^--P_0^- D$, where
$D$ belongs to the set $\{A_+,A_-,C_+,C_-\}$, are compact on the
space $L^p(\mR_+)$.
\end{lemma}
%%%----------------------------------------------------------------------------
\begin{proof}
Both representations follow from Lemma~\ref{le:PR-relations}.
The compactness of the commutators $DP_0^\pm-P_0^\pm D$ is a consequence of
parts (a) and (d) of Theorem~\ref{th:algebra-A} and
Lemma~\ref{le:compactness-commutators}.
\end{proof}
%%%----------------------------------------------------------------------------
The following statement generalizes \cite[Theorem 8.1]{KKL11a}
from the case of two-sided invertible binomial functional operators
$A_+$ and $A_-$ to the case of one-sided invertible operators
$A_+\in W_{\alpha,p}^{SO}$ and $A_-\in W_{\beta,p}^{SO}$. This generalization
is possible thanks to Theorem~\ref{th:one-sided-inverse-closedness},
although the proof follows the same lines as in \cite{KKL11a}.
%%%----------------------------------------------------------------------------
\begin{theorem}\label{th:suf-i}
Let $1<p<\infty$ and let $\gamma\in\mC$ satisfy \eqref{eq:gamma-condition}.
Suppose $a_k,b_k$ belong to $SO(\mR_+)$ for all $k\in\mZ$,
$\alpha,\beta$ belong to $SOS(\mR_+)$, the operators
$A_+\in W_{\alpha,p}^{SO}$ and $A_-\in W_{\beta,p}^{SO}$
are given by \eqref{eq:def-Apm}, and the operator $N$ is given by \eqref{eq:def-N}.
\begin{enumerate}
\item[{\rm (a)}]
If the operator $A_+$ is left (resp., right) invertible on the
space $L^p(\mR_+)$, then the coset $N^\pi+\cJ_{+\infty}^\pi$ is left (resp.,
right) invertible  in the quotient algebra $\Lambda_{+\infty}^\pi$.

\item[{\rm (b)}]
If the operator $A_-$ is left (resp., right) invertible on the
space $L^p(\mR_+)$, then the coset $N^\pi+\cJ_{-\infty}^\pi$ is left (resp.,
right) invertible  in the quotient algebra $\Lambda_{-\infty}^\pi$.
\end{enumerate}
\end{theorem}
%%%----------------------------------------------------------------------------
\begin{proof}
Recall that $\mathcal{FO}_{\alpha,\beta}\subset
\cF_{\alpha,\beta}\subset\Lambda$ in view of Theorem~\ref{th:embeddings}.
By Lemma~\ref{le:N-representations}, the operator $N$ is represented in
each of the forms \eqref{eq:suf-i-1}, where
\[
A_+\in W_{\alpha,p}^{SO}\subset\mathcal{FO}_{\alpha,\beta}
\subset\cF_{\alpha,\beta}\subset\Lambda,
\quad
A_-\in W_{\beta,p}^{SO}\subset\mathcal{FO}_{\alpha,\beta}
\subset\cF_{\alpha,\beta}\subset\Lambda,
\]
and
$C_+,C_-\in\cF_{\alpha,\beta}\subset\Lambda$.

(a) Take $A_+\in W_{\alpha,p}^{SO}$.
If $A_+$ is left (resp., right) invertible in $\cB$, then it
follows from Theorem~\ref{th:one-sided-inverse-closedness} that
there exists a left (resp., right) inverse $A_+^{(-1)}$ of $A_+$ such that
$A_+^{(-1)}\in W_{\alpha,p}^{SO}\subset
\mathcal{F}_{\alpha,\beta}\subset\Lambda$.
Hence the coset $A_+^\pi=A_++\cK$ is left (resp., right) invertible in the
quotient algebra $\Lambda^\pi$, which implies the left (resp., right)
invertibility of the coset $A_+^\pi+\cJ_{+\infty}^\pi$ in the quotient
algebra $\Lambda_{+\infty}^\pi$. Hence we infer from \eqref{eq:suf-i-1} that
%%%
\begin{align*}
N^\pi+\cJ_{+\infty}^\pi
&=
(A_+P_0^++C_-P_0^-)^\pi+
\cJ_{+\infty}^\pi
\\
&=
A_+^\pi+[(C_--A_+)P_0^-]^\pi+\cJ_{+\infty}^\pi
=
A_+^\pi+\cJ_{+\infty}^\pi
\end{align*}
%%%
because $(P_0^-)^\pi\in\cI_{+\infty}^\pi\subset\cJ_{+\infty}^\pi$.
Thus, the left (resp. right) invertibility of the operator $A_+$ in $\cB$
implies the left (resp., right) invertibility of the coset
$N^\pi+\cJ_{+\infty}^\pi$ in the quotient algebra $\Lambda^\pi_{+\infty}$.
Part (a) is proved.

(b) The proof is analogous.
\end{proof}
%%%----------------------------------------------------------------------------
\subsection{Invertibility in the quotient algebras
\boldmath{$\Lambda_{\xi,x}^\pi$} with \boldmath{$(\xi,x)\in\Delta\times\mR$}}
By a literal repetition with minor modifications of the proof of
\cite[Lemma~7.4]{KKL11a}, we get the following.
%%%----------------------------------------------------------------------------
\begin{lemma}\label{le:representations}
Let $\gamma\in\mC$ satisfy \eqref{eq:gamma-condition}. Suppose
$\alpha\in SOS(\mR_+)$ and $\omega$ is its exponent function. If
$\xi\in\Delta$ and
\[
\fa(t,x):=e^{i\omega(t)x}(r_\gamma(x))^2
\ \ \mbox{ for }\ \
(t,x)\in(\Delta\cup\mR_+)\times\mR,
\]
then there exists a function $\fb_\xi\in\cE(\mR_+,V(\mR))$ such that
\[
\fa(t,x)-\fa(\xi,x)=(\omega(t)-\omega(\xi))\fb_\xi(t,x)r_\gamma(x)
\ \ \mbox{ for }\ \
(t,x)\in\mR_+\times\mR.
\]
\end{lemma}
%%%----------------------------------------------------------------------------
\begin{lemma}\label{le:UR2}
Let $\gamma\in\mC$ satisfy \eqref{eq:gamma-condition}. Suppose $\alpha$ is a
slowly oscillating shift, $\omega$ is its exponent function, and $U_\alpha$ is
the associated isometric shift operator on $L^p(\mR_+)$. If
$(\xi,x)\in\Delta\times\mR$, then
\[
(U_\alpha R_\gamma^2)^\pi-e^{i\omega(\xi)x}(r_\gamma(x))^2 I^\pi
\in\cJ_{\xi,x}^\pi.
\]
\end{lemma}
%%%----------------------------------------------------------------------------
\begin{proof}
The proof is developed by analogy with \cite[Lemma~8.3]{KKL11a}. In view of
Lemma~\ref{le:shift-R-gamma},
%%%
\begin{equation}\label{eq:UR2-1}
U_\alpha R_\gamma\simeq\Phi^{-1}\Op(\fc_{\omega,\gamma})\Phi,
\end{equation}
%%%
where $\fc_{\omega,\gamma}\in\cE(\mR_+,V(\mR))$ is given by
\[
\fc_{\omega,\gamma}(t,y):=e^{i\omega(t)y}r_\gamma(y)
\ \ \mbox{ for }\ \
(t,y)\in\mR_+\times\mR.
\]
On the other hand, in view of Theorem~\ref{th:algebra-A}(d) and Lemma~\ref{le:ff-fs-fr},
%%%
\begin{equation}\label{eq:UR2-2}
R_\gamma=\Phi^{-1}\Co(r_\gamma)\Phi=\Phi^{-1}\Op(\fr_\gamma)\Phi,
\end{equation}
%%%
where $\fr_\gamma\in\cE(\mR_+,V(\mR))$ is given by
\[
\fr_\gamma(t,y)=r_\gamma(y)
\ \ \mbox{ for }\ \
(t,y)\in\mR_+\times\mR.
\]
It follows from \eqref{eq:UR2-1}--\eqref{eq:UR2-2} and
Lemma~\ref{le:PDO-3-operators} that
%%%
\begin{equation}\label{eq:UR2-3}
U_\alpha R_\gamma^2 \simeq \Phi^{-1}\Op(\fa)\Phi,
\end{equation}
%%%
where
\[
\fa(t,y)=\fc_{\omega,\gamma}(t,y)\fr_\gamma(t,y)=e^{i\omega(t)y}(r_\gamma(y))^2
\ \ \mbox{ for }\ \
(t,y)\in\mR_+\times\mR.
\]
Since $\fa(\cdot,y)\in SO(\mR_+)$ for every $y\in\mR$, taking the Gelfand
transform of $\fa(\cdot,y)$, we obtain
\[
\fa(t,y)=e^{i\omega(t)y}(r_\gamma(y))^2
\ \ \mbox{ for }\ \
(t,y)\in(\Delta\cup\mR_+)\times\mR.
\]
Fix $(\xi,x)\in\Delta\times\mR$. Let us represent the function $\fa$ in the
form
%%%
\begin{align}
\fa(t,y)
=&
\fa(t,y)-\fa(\xi,y)+\fc_{\omega,\gamma}(\xi,y)r_\gamma(y)
\nonumber\\
=&
[\fa(t,y)-\fa(\xi,y)]+[\fc_{\omega,\gamma}(\xi,y)-\fc_{\omega,\gamma}(\xi,x)]\fr_\gamma(t,y)
\nonumber\\
&+
\fc_{\omega,\gamma}(\xi,x)[\fr_\gamma(t,y)-\fr_\gamma(t,x)]+\fa(\xi,x),
\label{eq:UR2-4}
\end{align}
%%%
where $(t,y)\in\mR_+\times\mR$. Further, we deduce from
Lemma~\ref{le:representations} that there exists a function
$\fb_\xi\in\cE(\mR_+,V(\mR))$ such that
\[
\fa-\fa(\xi,\cdot)=(\omega-\omega(\xi))\fb_\xi\fr_\gamma.
\]
Hence, we infer from the above equality and Lemmas~\ref{le:ff-fs-fr} and
\ref{le:PDO-3-operators} that
%%%
\begin{align*}
\Phi^{-1}\Op(\fa &-\fa(\xi,\cdot))\Phi
=
\Phi^{-1}\Op((\omega-\omega(\xi))\fb_\xi\fr_\gamma)\Phi
\nonumber\\
&=
(\Phi^{-1}\Op(\omega-\omega(\xi))\Phi)\cdot
(\Phi^{-1}\Op(\fb_\xi)\Phi)\cdot
(\Phi^{-1}\Op(\fr_\gamma)\Phi).
\end{align*}
%%%
The latter equality, Theorem~\ref{th:comp-semi-commutators-PDO} and equality
\eqref{eq:UR2-2} imply that
%%%
\begin{equation}\label{eq:UR2-5}
\Phi^{-1}\Op(\fa -\fa(\xi,\cdot))\Phi
\simeq(\omega-\omega(\xi))R_\gamma(\Phi^{-1}\Op(\fb_\xi)\Phi).
\end{equation}
%%%
Applying Theorem~\ref{th:algebra-Zpi}(c) and
Corollary~\ref{co:algebra-Api}(b), we obtain
\[
([(\omega-\omega(\xi))R_\gamma]^\pi)\widehat{\hspace{2mm}}(\xi,x)
=
(\omega(\xi)-\omega(\xi))r_\gamma(x)=0.
\]
Therefore $[(\omega-\omega(\xi))R_\gamma]^\pi\in\cI_{\xi,x}^\pi$.
On the other hand, since $\fb_\xi\in\cE(\mR_+,V(\mR))$, we conclude from
Lemma~\ref{le:PDO-in-Lambda} that $\Phi^{-1}\Op(\fb_\xi)\Phi\in\Lambda$.
Then, taking into account \eqref{eq:UR2-5} and the definition of the ideal
$\cJ_{\xi,x}^\pi$, we infer that
%%%
\begin{equation}\label{eq:UR2-6}
[\Phi^{-1}\Op(\fa-\fa(\xi,\cdot))\Phi]^\pi
=
[(\omega-\omega(\xi))R_\gamma(\Phi^{-1}\Op(\fb_\xi)\Phi)]^\pi
\in\cJ_{\xi,x}^\pi.
\end{equation}
%%%
Taking into account the definition of the norm \eqref{eq:V-norm} in the algebra
$V(\mR)$, it is easy to see that the function $\fc_{\omega,\gamma}(\xi,\cdot)$
belongs to $V(\mR)$, where
\[
\fc_{\omega,\gamma}(\xi,y)=e^{i\omega(\xi)y}r_\gamma(y)
\ \ \mbox{ for }\ \
y\in\mR_+.
\]
Then, by Stechkin's inequality (Theorem~\ref{th:Stechkin}),
$\fc_{\omega,\gamma}(\xi,\cdot)\in C_p(\overline{\mR})$. Hence, it follows
from Theorem~\ref{th:algebra-A}(a) that
$\Phi^{-1}\Co(\fc_{\omega,\gamma}(\xi,\cdot))\Phi\in\cA$. By
Theorem~\ref{th:algebra-Zpi}(c) and Corollary~\ref{co:algebra-Api}(b),
%%%
\begin{align*}
([\Phi^{-1}\Op(\fc_{\omega,\gamma}(\xi,\cdot)
&-
\fc_{\omega,\gamma}(\xi,x))\Phi]^\pi)\widehat{\hspace{2mm}}(\xi,x)=
\\
&=
([\Phi^{-1}\Co(\fc_{\omega,\gamma}(\xi,\cdot))\Phi]^\pi)
\widehat{\hspace{2mm}}(\xi,x)
-
([\fc_{\omega,\gamma}(\xi,x)I]^\pi)\widehat{\hspace{2mm}}(\xi,x)
\\
&=
e^{i\omega(\xi)x}r_\gamma(x)-e^{i\omega(\xi)x}r_\gamma(x)=0.
\end{align*}
%%%
Therefore
\[
[\Phi^{-1}\Op(\fc_{\omega,\gamma}(\xi,\cdot)
-\fc_{\omega,\gamma}(\xi,x))\Phi]^\pi\in\cI_{\xi,x}^\pi.
\]
By this observation, Lemma~\ref{le:PDO-3-operators} and equality
\eqref{eq:UR2-2}, we obtain
%%%
\begin{align}
[\Phi^{-1}\Op((\fc_{\omega,\gamma}(\xi,\cdot)
&-
\fc_{\omega,\gamma}(\xi,x))\fr_\gamma)\Phi]^{\pi}=
\nonumber\\
&=
[\Phi^{-1}\Op(\fc_{\omega,\gamma}(\xi,\cdot)
-\fc_{\omega,\gamma}(\xi,x))\Phi]^\pi
[\Phi^{-1}\Op(\fr_\gamma)\Phi]^\pi
\nonumber\\
&=
[\Phi^{-1}\Op(\fc_{\omega,\gamma}(\xi,\cdot)
-\fc_{\omega,\gamma}(\xi,x))\Phi]^\pi
R_\gamma^\pi\in\cJ_{\xi,x}^\pi.
\label{eq:UR2-7}
\end{align}
%%%
Finally, in view of \eqref{eq:UR2-2}, Theorem~\ref{th:algebra-Zpi}(c) and
Corollary~\ref{co:algebra-Api}(b), we deduce that
\[
([\Phi^{-1}\Op(\fr_\gamma-r_\gamma(x))\Phi]^\pi)\widehat{\hspace{2mm}}(\xi,x)
=
([R_\gamma-r_\gamma(x)I]^\pi)\widehat{\hspace{2mm}}(\xi,x)
=
r_\gamma(x)-r_\gamma(x)=0.
\]
Hence
%%%
\begin{equation}\label{eq:UR2-8}
\fc_{\omega,\gamma}(\xi,x)[\Phi^{-1}\Op(\fr_\gamma-r_\gamma(x))\Phi]^\pi
\in\cI_{\xi,x}^\pi\subset\cJ_{\xi,x}^\pi.
\end{equation}
%%%
Combining \eqref{eq:UR2-3}--\eqref{eq:UR2-4} with
\eqref{eq:UR2-6}--\eqref{eq:UR2-8}, we arrive at the relation
\[
(U_\alpha R_\gamma^2)^\pi-e^{i\omega(\xi)x}(r_\gamma(x))^2I^\pi
=
[\Phi^{-1}\Op(\fa)\Phi-\fa(\xi,x)I]^\pi\in\cJ_{\xi,x}^\pi,
\]
which completes the proof.
\end{proof}
%%%----------------------------------------------------------------------------
Now we are in a position to prove that condition (ii)
of Theorem~\ref{th:one-sided-Fredholmness} is sufficient for the
invertibility of the coset $N^\pi+\cJ_{\xi,x}^\pi$ in the quotient algebra
$\Lambda_{\xi,x}^\pi$.
%%%----------------------------------------------------------------------------
\begin{theorem}\label{th:suf-ii}
Let $1<p<\infty$ and let $\gamma\in\mC$ satisfy \eqref{eq:gamma-condition}.
Suppose $a_k,b_k$ belong to $SO(\mR_+)$ for all $k\in\mZ$, $\alpha,\beta$
belong to $SOS(\mR_+)$, the operators $A_+\in W_{\alpha,p}^{SO}$ and
$A_-\in W_{\beta,p}^{SO}$ are given by \eqref{eq:def-Apm}, and the operator
$N$ is given by \eqref{eq:def-N}. If $n(\xi,x)\ne 0$ for some
$(\xi,x)\in\Delta\times\mR$, where the function $n$ is defined by
\eqref{eq:def-apm}--\eqref{eq:def-n}, then the coset $N^\pi+\cJ_{\xi,x}^\pi$
is two-sided invertible in the quotient algebra $\Lambda_{\xi,x}^\pi$.
\end{theorem}
%%%----------------------------------------------------------------------------
\begin{proof}
We follow the main lines of the proof of \cite[Theorem~8.4]{KKL11a}.

Fix $(\xi,x)\in\Delta\times\mR$ and consider the operators
%%%
\begin{equation}\label{eq:suf-ii-1}
H_\pm:=\frac{p_\gamma^\pm(x)}{[r_\gamma(x)]^2}R_\gamma^2.
\end{equation}
%%%
Then it follows from Theorem~\ref{th:algebra-Zpi}(c) and Corollary~\ref{co:algebra-Api}(b) that
%%%
\begin{equation}\label{eq:Hpm}
(H_\pm^\pi)\widehat{\hspace{2mm}}(\xi,x)=p_\gamma^\pm(x).
\end{equation}
%%%
Therefore, taking into account Corollary~\ref{co:algebra-Api}(b)
once again, we get
\[
(P_\gamma^\pm-H_\pm)^\pi\in\cI_{\xi,x}^\pi
\]
and
\[
\big[
A_+(P_\gamma^+-H_+)+A_-(P_\gamma^--H_-)
\big]^\pi\in\cJ_{\xi,x}^\pi,
\]
whence
%%%
\begin{equation}\label{eq:suf-ii-2}
N^\pi+\cJ_{\xi,x}^\pi=
(A_+H_++A_-H_-)^\pi
+\cJ_{\xi,x}^\pi.
\end{equation}
%%%
We know from Theorem~\ref{th:algebra-A}(d),(a) that $H_\pm\in\cA$.
Hence, we infer from Lemma~\ref{le:compactness-commutators} that
for all $k\in\mZ$,
%%%
\begin{equation}\label{eq:suf-ii-3}
(U_\alpha^k H_+)^\pi=(H_+U_\alpha^k)^\pi,
\quad
(U_\beta^k H_-)^\pi=(H_-U_\beta^k)^\pi.
\end{equation}
%%%
Taking into account \eqref{eq:Hpm}, it is easy to see that
for all $k\in\mZ$,
\[
\big[(a_kH_+)^\pi-(a_k(\xi)H_+)^\pi\big]\widehat{\hspace{2mm}}(\xi,x)=0,
\quad
\big[(b_kH_-)^\pi-(b_k(\xi)H_-)^\pi\big]\widehat{\hspace{2mm}}(\xi,x)=0.
\]
%%%
Hence
%%%
\begin{equation}\label{eq:suf-ii-4}
(a_kH_+)^\pi-(a_k(\xi)H_+)^\pi,\
(b_kH_-)^\pi-(b_k(\xi)H_-)^\pi
\in\cI_{\xi,x}^\pi\subset\cJ_{\xi,x}^\pi,
\end{equation}
%%%
Applying \eqref{eq:suf-ii-3} and \eqref{eq:suf-ii-4}, we obtain
for all $k\in\mZ$,
%%%
\begin{align}
(a_kU_\alpha^k H_+)^\pi-(a_k(\xi)U_\alpha^k H_+)^\pi
&=
[(a_kH_+-a_k(\xi)H_+)U_\alpha^k]^\pi\in\cJ_{\xi,x}^\pi,
\label{eq:suf-ii-5}
\\
(b_kU_\beta^k H_-)^\pi-(b_k(\xi)U_\beta^k H_-)^\pi
&=
[(b_kH_--b_k(\xi)H_-)U_\beta^k]^\pi\in\cJ_{\xi,x}^\pi.
\label{eq:suf-ii-6}
\end{align}
%%%
Then it follows from \eqref{eq:suf-ii-2} and
\eqref{eq:suf-ii-5}--\eqref{eq:suf-ii-6} that
%%%
\begin{equation}\label{eq:suf-ii-7}
N^\pi+\cJ_{\xi,x}^\pi=
\sum_{k\in\mZ}
\big(a_k(\xi)U_\alpha^k H_++b_k(\xi)U_\beta^k H_-\big)^\pi+\cJ_{\xi,x}^\pi.
\end{equation}
%%%
In view of Theorem~\ref{th:algebra-Zpi}(c) and
Corollary~\ref{co:algebra-Api}(b), it is easy to see that
\[
\big(\big[(r_\gamma(x))^{-2}R_\gamma^2-I\big]^\pi\big)\widehat{\hspace{2mm}}(\xi,x)=0.
\]
Hence
%%%
\begin{equation}\label{eq:suf-ii-8}
H_\pm^\pi-p_\gamma^\pm(x)I^\pi\in\cI_{\xi,x}^\pi\subset\cJ_{\xi,x}^\pi.
\end{equation}
%%%
By Lemmas~\ref{le:SOS-iterations}, \ref{le:exponent-function-iterations}, and
\ref{le:UR2}, we deduce for all $k\in\mZ$ that
\[
(U_\alpha^k R_\gamma^2)^\pi-e^{ik\omega(\xi)x}(r_\gamma(x))^2I^\pi,
\
(U_\beta^k R_\gamma^2)^\pi-e^{ik\eta(\xi)x}(r_\gamma(x))^2I^\pi\in\cJ_{\xi,x}^\pi.
\]
The above inclusions together with \eqref{eq:suf-ii-1} imply for every
$k\in\mZ$ that
%%%
\begin{align}
&
(U_\alpha^k H_+)^\pi-e^{ik\omega(\xi)x}p_\gamma^+(x)\,I^\pi\in\cJ_{\xi,x}^\pi,
\label{eq:suf-ii-9}
\\
&
(U_\beta^k H_-)^\pi-e^{ik\eta(\xi)x}p_\gamma^-(x)\,I^\pi\in\cJ_{\xi,x}^\pi.
\label{eq:suf-ii-10}
\end{align}
%%%
Combining \eqref{eq:suf-ii-8}--\eqref{eq:suf-ii-10}, we arrive at the equality
\[
N^\pi+\cJ_{\xi,x}^\pi=n(\xi,x)I^\pi+\cJ_{\xi,x}^\pi,
\]
where $n(\xi,x)$ is given by \eqref{eq:def-p-gamma-pm}--\eqref{eq:def-n}.
If $n(\xi,x)\ne 0$, then one
can check straightforwardly that $(1/n(\xi,x))I^\pi+\cJ_{\xi,x}^\pi$ is the
inverse of the coset $N^\pi+\cJ_{\xi,x}^\pi$ in the quotient algebra
$\Lambda_{\xi,x}^\pi$.
\end{proof}
%%%----------------------------------------------------------------------------
\subsection{Proof of Theorem~{${\bf\ref{th:one-sided-Fredholmness}}$}}
The proof is analogous to that of \cite[Theorem~1.2]{KKL11a}. We know from
Theorem~\ref{th:algebra-A}(d) and Theorem~\ref{th:embeddings} that
$N\in\Lambda$. If condition (i) of Theorem~\ref{th:one-sided-Fredholmness} is
fulfilled, that is, if the operators $A_+$ and $A_-$
are left (resp., right) invertible, then by Theorem~\ref{th:suf-i}
the coset $N^\pi+\cJ_{+\infty}^\pi$ is left (resp., right) invertible
in the quotient algebra $\Lambda_{+\infty}^\pi$
and the coset $N^\pi+\cJ_{-\infty}^\pi$
is left (resp., right) invertible  in the quotient algebra
$\Lambda_{-\infty}^\pi$. On the other hand, if condition (ii) of
Theorem~\ref{th:one-sided-Fredholmness} holds, then in view of
Theorem~\ref{th:suf-ii}, the coset $N^\pi+\cJ_{\xi,x}^\pi$ is two-sided
invertible in the quotient algebra $\Lambda_{\xi,x}^\pi$ for every pair
$(\xi,x)\in\Delta\times\mR$. Then, by Theorem~\ref{th:localization-left-right},
the operator $N\in\Lambda$ is left (resp., right) Fredholm.
\qed
%%%----------------------------------------------------------------------------
\section{Semi-Fredholmness of weighted singular integral operators with
coefficients being binomial functional operators}
\label{sec:binomial}
\subsection{Criteria for the two-sided and strict one-sided invertibility of
\boldmath{$aI-bU_\alpha$}}
Suppose $a,b\in SO(\mR_+)$ and $\alpha\in SOS(\mR_+)$. For $s\in\{0,\infty\}$, put
\[
L_*(s;a,b):=\liminf\limits_{t\to s}(|a(t)|-|b(t)|),
\quad
L^*(s;a,b):=\limsup\limits_{t\to s}(|a(t)|-|b(t)|).
\]
Fix a point $\tau\in\mR_+$ and put
\[
\tau_{-,\alpha}:=\lim_{n\to-\infty}\alpha_n(\tau),
\quad
\tau_{+,\alpha}:=\lim_{n\to+\infty}\alpha_n(\tau).
\]
Then either $\tau_{-,\alpha}=0$ and
$\tau_{+,\alpha}=\infty$, or $\tau_{-,\alpha}=\infty$ and $\tau_{+,\alpha}=0$.
The points $\tau_{+,\alpha}$ and $\tau_{-,\alpha}$
are called attracting and repelling points of
$\alpha$, respectively.

We say that the triple
$\{\alpha,a,b\}$ satisfies conditions
(I1), (I2), (LI), (RI) if
\begin{enumerate}
\item[(I1)]
$L_*(\tau_{-,\alpha};a,b)>0$ and
$L_*(\tau_{+,\alpha};a,b)>0$ and
$\inf\limits_{t\in\mR_+}|a(t)|>0$;

\item[(I2)]
$L^*(\tau_{-,\alpha};a,b)<0$ and
$L^*(\tau_{+,\alpha};a,b)<0$ and
$\inf\limits_{t\in\mR_+}|b(t)|>0$;

\item[(LI)]
$L^*(\tau_{-,\alpha};a,b)<0<L_*(\tau_{+,\alpha};a,b)$
and for every $t\in\mR_+$ there is an integer $k_t$ such that
$b[\alpha_k(t)]\ne 0$ for $k<k_t$ and $a[\alpha_k(t)]\ne 0$ for $k>k_t$.

\item[(RI)]
$L^*(\tau_{+,\alpha};a,b)<0<L_*(\tau_{-,\alpha};a,b)$
and for every $t\in\mR_+$ there is an integer $k_t$ such that
$b[\alpha_k(t)]\ne 0$ for $k\ge k_t$ and $a[\alpha_k(t)]\ne 0$ for $k<k_t$.
\end{enumerate}
%%%----------------------------------------------------------------------------
\begin{theorem}[{\cite[Theorems~1.1--1.2]{KKL-MJOM}}]
\label{th:FO-binomial}
Let $a,b\in SO(\mR_+)$, $\alpha\in SOS(\mR_+)$, and let
the binomial functional operator $A$ be given by
\[
A:=aI-bU_\alpha.
\]
\begin{enumerate}
\item[(a)]
The operator $A$ is invertible on the Lebesgue space $L^p(\mR_+)$ if
and only if the triple $\{\alpha,a,b\}$  satisfies either condition {\rm(I1)},
or condition {\rm(I2)}.

\item[(b)]
The operator $A$ is strictly left invertible on the space $L^p(\mR_+)$
if and only if the triple $\{\alpha,a,b\}$ satisfies condition {\rm(LI)}.

\item[(c)]
The operator $A$ is strictly right invertible on the space $L^p(\mR_+)$
if and only if the triple $\{\alpha,a,b\}$ satisfies condition {\rm(RI)}.
\end{enumerate}
\end{theorem}
%%%----------------------------------------------------------------------------
\subsection{Sufficient conditions for the semi-Fredholmness}
Combining Theorem~\ref{th:one-sided-Fredholmness} and
Theorem~\ref{th:FO-binomial}, we arrive at the following.
%%%---------------------------------------------------------------------------
\begin{corollary}\label{co:sufficiency-binomial}
Let $1<p<\infty$ and let $\gamma\in\mC$ satisfy \eqref{eq:gamma-condition}.
Suppose $a,b,c,d$ belong to $SO(\mR_+)$, $\alpha,\beta$ belong to $SOS(\mR_+)$,
and $\omega,\eta\in SO(\mR_+)$ are the exponent functions of the shifts
$\alpha,\beta$, respectively. Consider the operator
\[
M:=(aI-bU_\alpha)P_\gamma^++(cI-dU_\beta)P_\gamma^-
\]
and the corresponding function $m$ defined for $(\xi,x)\in(\mR_+\cup\Delta)\times\mR$
by
\[
m(\xi,x):=
(a(\xi)-b(\xi)e^{i\omega(\xi)x})p_\gamma^+(x)
+
(c(\xi)-d(\xi)e^{i\eta(\xi)x})p_\gamma^-(x),
\]
where the functions $p_\gamma^\pm$ are defined by \eqref{eq:def-p-gamma-pm}.
%%%
\begin{enumerate}
\item[(a)]
If each of the triples $\{\alpha,a,b\}$ and $\{\beta,c,d\}$
satisfies either condition {\rm(I1)} or condition {\rm(I2)}
(but not necessarily the same condition), and
%%%
\begin{equation}\label{eq:m-non-degeneracy}
\inf_{x\in\mR}|m(\xi,x)|>0
\quad\mbox{for every}\quad
\xi\in\Delta,
\end{equation}
%%%
then the operator $M$ is Fredholm on the space $L^p(\mR_+)$.

\item[(b)]
If each of the triples $\{\alpha,a,b\}$ and $\{\beta,c,d\}$
satisfies only one  of conditions
{\rm(I1)}, {\rm(I2)} and {\rm(LI)} (but not necessarily the same condition),
and condition \eqref{eq:m-non-degeneracy} is
fulfilled, then the operator $M$ is left Fredholm on the space $L^p(\mR_+)$.

\item[(c)]
If each of the triples $\{\alpha,a,b\}$ and $\{\beta,c,d\}$
satisfies only one  of conditions
{\rm(I1)}, {\rm(I2)} and {\rm(RI)} (but not necessarily the same condition),
and condition \eqref{eq:m-non-degeneracy} is
fulfilled, then the operator $M$ is right Fredholm on the space $L^p(\mR_+)$.
\end{enumerate}
\end{corollary}
%%%----------------------------------------------------------------------------
Another (more involved) proof of Corollary~\ref{co:sufficiency-binomial}(a),
relying on criteria for the Fredholmness of Mellin pseudodifferential
operators (see \cite{K06} and \cite[Theorem~3.6]{KKL14}), is given in
\cite[Theorem~1.3]{KKL-JMAA}. The converse statement to 
Corollary~\ref{co:sufficiency-binomial}(a) is proved in 
\cite[Theorem~1.2]{KKL-JIEA}. The statements of parts (b) and (c) in
Corollary~\ref{co:sufficiency-binomial} are new.
%%%----------------------------------------------------------------------------
\subsection*{Acknowledgment}
We would like to thank the anonymous referee for pointing out that
Lemma~\ref{le:Hilbert-one-sided} can be obtained from 
\cite[Example~2.16]{HRS01}.
%%%----------------------------------------------------------------------------


\begin{thebibliography}{00}
\bibitem{BKS02}
A. B\"ottcher, Yu. I. Karlovich, and I. M. Spitkovsky,
\textit{Convolution Operators and Factorization of Almost Periodic Matrix
Functions}.
Operator Theory: Advances and Applications, vol. 131.
Birkh\"auser, Basel, 2002.

\bibitem{BS99}
A. B\"ottcher and B. Silbermann,
\textit{Introduction to Large Truncated Toeplitz Matrices}.  
Springer, New York, 1999.

\bibitem{BS06}
A. B\"ottcher and B. Silbermann,
\textit{Analysis of Toeplitz Operators}.
Springer, Berlin, 2006.

\bibitem{C90}
J. B. Conway,
\textit{A Course in Functional Analysis}.
Springer, Berlin, 1990.

\bibitem{DS08}
V. D. Didenko and B. Silbermann,
\textit{Approximation of Additive Convolution-Like Operators.
Real C*-Algebra Approach}.
Birkh\"auser, Basel, 2008.

\bibitem{D79}
R. Duduchava,
\textit{Integral Equations with Fixed Singularities}.
Teubner Verlagsgesellschaft, Leipzig, 1979.

\bibitem{D87}
R. Duduchava,
\textit{On algebras generated by convolutions and discontinuous functions}.
Integr. Equ. Oper. Theory \textbf{10} (1987) 505--530.

\bibitem{FTK-BJMA}
G. Fern{\'a}ndez-Torres and Yu. I. Karlovich,
\textit{Two-sided and one-sided invertibility of the Wiener type functional
operators with a shift and slowly oscillating data}.
Banach J. Math. Analysis, to appear.
doi: 10.1215/17358787-2017-0006.

\bibitem{FTK}
G. Fern{\'a}ndez-Torres and Yu. I. Karlovich,
\textit{Fredholmness of nonlocal singular integral operators with slowly
oscillating data}.
Proceedings of the 14th International Conference on Integral Methods
in Science and Engineering (IMSE 2016), to appear.

\bibitem{GK92}
I. Gohberg and N. Krupnik,
\textit{One-Dimensional Linear Singular Integral Equations. I. Introduction}.
Operator Theory: Advances and Applications, vol. 53.
Birkh\"auser, Basel, 1992.

\bibitem{G70}
V. V. Grushin,
\textit{Pseudodifferential operators on $\mR^n$ with bounded symbols}.
Funct. Anal. Appl. \textbf{4} (1970) 202--212.

\bibitem{HRS94}
R. Hagen, S. Roch, and B. Silbermann,
\textit{Spectral Theory of Approximation Methods for Convolution Equations}.
Operator Theory: Advances and Applications, vol. 74.
Birkh\"auser, Basel, 1994.

\bibitem{HRS01}
R. Hagen, S. Roch, and B. Silbermann,
\textit{C*-Algebras and Numerical Analysis}. 
Pure and Applied Mathematics, vol. 236. 
Marcel Dekker, New York, 2001.

\bibitem{K15}
A. Yu. Karlovich,
\textit{Fredholmness and index of simplest weighted singular integral
operators with two slowly oscillating shifts}.
Banach J. Math. Anal. \textbf{9} (2015) 24--42.

\bibitem{KKL11a}
A. Yu. Karlovich, Yu. I. Karlovich, and A. B. Lebre,
\textit{Sufficient conditions for Fredholmness of singular integral
operators with shifts and slowly oscillating data}.
Integr. Equ. Oper. Theory \textbf{70} (2011) 451--483.

\bibitem{KKL14}
A. Yu. Karlovich, Yu. I. Karlovich, and A. B. Lebre,
\textit{Fredholmness and index of simplest singular integral operators
with two slowly oscillating shifts}.
Operators and Matrices \textbf{8} (2014) 935--955.

\bibitem{KKL16}
A. Yu. Karlovich, Yu. I. Karlovich, and A. B. Lebre,
\textit{On a weighted singular integral operator with shifts and slowly oscillating data}.
Complex Anal. Oper. Theory \textbf{10} (2016) 1101--1131.

\bibitem{KKL-MJOM}
A. Yu. Karlovich, Yu. I. Karlovich, and A. B. Lebre,
\textit{One-sided invertibility criteria for binomial functional operators
with shift and slowly oscillating data}.
Mediterr. J. Math. \textbf{13} (2016) 4413--4435.

\bibitem{KKL-JMAA}
A. Yu. Karlovich, Yu. I. Karlovich, and A. B. Lebre,
\textit{The index of weighted singular integral operators with shifts and slowly oscillating data}.
J. Math. Anal. Appl. \textbf{450} (2017), 606--630.

\bibitem{KKL-JIEA}
A. Yu. Karlovich, Yu. I. Karlovich, and A. B. Lebre,
\textit{Necessary Fredholm conditions for weighted singular integral operators
with shifts and slowly oscillating data}.
J. Integr. Equ. Appl., to appear.

\bibitem{K06}
Yu. I. Karlovich,
\textit{An algebra of pseudodifferential operators with slowly oscillating symbols}.
Proc. London Math. Soc. \textbf{92} (2006) 713--761.

\bibitem{K08}
Yu. I. Karlovich,
\textit{Nonlocal singular integral operators with slowly oscillating data}.
In ``Operator Algebras, Operator Theory and Applications",
eds. M. A. Bastos et al.
Operator Theory: Advances and Applications, vol. 181, 2008, pp.~229--261.

\bibitem{K82}
H. Kumano-go,
\textit{Pseudo-Differential Operators}.
The MIT Press. Cambridge, MA, 1982.

\bibitem{N72}
M. A. Naimark,
\textit{Normed Algebras}.
Wolters-Noordhoff Publishing, Groningen, 1972.

\bibitem{N55}
I. P. Natanson,
\textit{Theory of Functions of a Real Variable}.
Frederick Ungar Publishing Co., New York, 1955.

\bibitem{R92}
V. S. Rabinovich,
\textit{Singular integral operators on a composed contour with oscillating
tangent and pseudodifferential Mellin operators}.
Soviet Math. Dokl. \textbf{44} (1992) 791--796.

\bibitem{R98}
V. S. Rabinovich,
\textit{Mellin pseudodifferential operators techniques in the theory of
singular integral operators on some Carleson curves}.
In ``Differential and Integral Operators (Regensburg, 1995)",
eds. I. Gohberg et al.
Operator Theory: Advances and Applications, vol. 102, 1998, pp. 201--218.

\bibitem{RRS04}
V. S. Rabinovich, S. Roch, and B. Silbermann,
\textit{Limit Operators and Their Applications in Operator Theory}.
Operator Theory: Advances and Applications, vol. 150.
Birkh\"auser, Basel, 2004.

\bibitem{RSS11}
S. Roch, P. A. Santos, and B. Silbermann,
\textit{Non-Commutative Gelfand Theories. A Tool-kit for Operator Theorists and Numerical Analysts}.
Springer, Berlin, 2011.

\bibitem{S77}
D. Sarason,
\textit{Toeplitz operators with piecewise quasicontinuous symbols}.
Indiana Univ. Math. J. \textbf{26} (1977) 817--838.

\bibitem{SCNM86}
I. B. Simonenko and Chin Ngok Minh,
\textit{Local Approach to the Theory of One-Dimensional Singular Integral
Equations with Piecewise Continuous Coefficients. Noethericity}.
Rostov University Press, Rostov-on-Don, 1986 (in Russian).

\bibitem{St93}
E. M. Stein,
\textit{Harmonic Analysis: Real-Variable Methods, Orthogonality, and Oscillatory Integrals}.
Princeton Univ. Press, Princeton, NJ, 1993.

\bibitem{Y80}
K. Yosida,
\textit{Functional Analysis}. 6th ed.
Springer, Berlin, 1980.
\end{thebibliography}
\end{document}